\documentclass{aims}
\usepackage{amsmath}
  \usepackage{paralist}
  \usepackage{graphics} 
  \usepackage{epsfig} 
\usepackage{graphicx}  \usepackage{epstopdf} 
 \usepackage[colorlinks=true]{hyperref}
\hypersetup{urlcolor=blue, citecolor=red}

\def\currentvolume{X}
 \def\currentissue{X}
  \def\currentyear{200X}
   \def\currentmonth{XX}
    \def\ppages{X--XX}

\newtheorem{theorem}{Theorem}[section]
\newtheorem{corollary}[theorem]{Corollary}

\newtheorem{lemma}[theorem]{Lemma}
\newtheorem{proposition}[theorem]{Proposition}

\theoremstyle{definition}
\newtheorem{definition}[theorem]{Definition}
\newtheorem{assumption}[theorem]{Assumption}
\newtheorem{example}[theorem]{Example}
\newtheorem{remark}[theorem]{Remark}

\usepackage{amssymb}
\usepackage{mathtools}
\newcommand{\R}{\mathbb{R}}
\newcommand{\N}{\mathbb{N}}
\newcommand{\W}{\mathbb{W}}

\newcommand{\Bkal}{\mathcal{B}}
\newcommand{\Fkal}{\mathcal{F}}  
\newcommand{\Hkal}{\mathcal{H}}
\newcommand{\Lkal}{\mathcal{L}}
\newcommand{\Ckal}{\mathcal{C}}
\newcommand{\Ekal}{\mathcal{E}}
\newcommand{\Rkal}{\mathcal{R}}
\newcommand{\Gkal}{\mathcal{G}}
\DeclareMathOperator{\Dom}{Dom} 
\DeclareMathOperator{\tr}{tr} 
\DeclareMathOperator{\id}{id} 
\DeclareMathOperator{\dist}{dist} 
\DeclareMathOperator{\supp}{supp} 
\newcommand{\dd}{\,\textnormal{d}}
\newcommand{\embeds}{\hookrightarrow}
\newcommand{\blangle}{\bigl\langle}
\newcommand{\brangle}{\bigr\rangle}
\newcommand{\eps}{\varepsilon}

\usepackage[style=numeric,sortcites=true,abbreviate=true,giveninits=true,isbn=false,url=false,eprint=false]{biblatex}  
\title[Quasilinear parabolic PDE{\tiny s} in the Bessel dual scale] 
      {Global-in-time solutions for quasilinear parabolic PDE{\scriptsize s}
        with mixed boundary conditions in the Bessel dual scale} 

\author[Fabian Hoppe, Hannes Meinlschmidt, and Ira Neitzel]{}


 \email{fabian.hoppe@dlr.de}
 \email{meinlschmidt@math.fau.de}
 \email{neitzel@ins.uni-bonn.de}

\thanks{This research was carried out while F.H.\ was affiliated with University of Bonn and partially supported by the Deutsche Forschungsgemeinschaft (DFG, German Research Foundation)--Projektnummer~211504053--SFB~1060. I.N.\ gratefully acknowledges financial support by the Deutsche Forschungsgemeinschaft (DFG, German Research Foundation)--Projektnummer~211504053--SFB~1060.}

\thanks{$^*$Corresponding author: Hannes Meinlschmidt}

\begin{document}
\maketitle

\centerline{\scshape Fabian Hoppe}
\medskip
{\footnotesize
 \centerline{Deutsches Zentrum f\"ur Luft- und Raumfahrt DLR}
   \centerline{Institut f\"ur Softwaretechnologie, High Performance Computing,}
   \centerline{Linder H\"ohe, 51147 K\"oln, Germany}
} 

\medskip

\centerline{\scshape Hannes Meinlschmidt}
\medskip
{\footnotesize
 \centerline{Department of Data Science (DDS),}
   \centerline{Chair in Dynamics, Control and Numerics (Alexander von Humboldt-Professorship),}
   \centerline{Friedrich-Alexander-Universit\"at Erlangen-N\"urnberg,}
   \centerline{Cauerstra{\ss}e 11, 91058 Erlangen, Germany}
}

\medskip

\centerline{\scshape Ira Neitzel}
\medskip
{\footnotesize
 \centerline{Institut f\"ur Numerische Simulation,}
   \centerline{Rheinische Friedrich-Wilhelms-Universit\"at Bonn,}
   \centerline{Friedrich-Hirzebruch-Allee 7, 53115 Bonn, Germany}
}

\bigskip

 \centerline{(Communicated by the associate editor name)}

\begin{abstract}
  We prove existence and uniqueness of global-in-time solutions in the
  $W^{-1,p}_D$-$W^{1,p}_D$-setting for abstract quasilinear parabolic PDEs with nonsmooth data
  and mixed boundary conditions, including a nonlinear source term with at most linear
  growth. Subsequently, we use a bootstrapping argument to achieve improved regularity of these
  global-in-time solutions within the functional-analytic setting of the interpolation scale of
  Bessel-potential dual spaces $H^{\theta-1,p}_D = [W^{-1,p}_D,L^p]_\theta$ with $\theta \in [0,1]$
  for the abstract equation under suitable additional assumptions. This is done by means of new
  nonautonomous maximal parabolic regularity results for nonautonomous differential operators
  operators with H\"older-continuous coefficients on Bessel-potential spaces. The upper limit
  for $\theta$ is derived from the maximum degree of H\"older continuity for solutions to an
  elliptic mixed boundary value problem in $L^p$.
\end{abstract}

\medskip
Received xxxx 20xx; revised xxxx 20xx; early access xxxx 20xx.
\medskip

\section{Introduction}

This work is concerned with global-in-time existence of solutions $u \colon (0,T) \times \Omega \to \R$ to quasilinear
parabolic equations of type
\begin{equation}\label{eq::concrete}
\left.  \begin{aligned} \partial_t u - \operatorname{div}(\xi(u)\mu
      \nabla u) + u &= \Fkal_\Omega(u) \qquad& &\text{on } (0,T) \times
      \Omega, \\ 
      \nu_{\partial \Omega} \cdot \xi(u) \mu \nabla u + \alpha  u &=
      \Fkal_\Gamma(u) & &\text{on } (0,T) \times \Gamma_N, \\ 
      u &= 0 &&\text{on } (0,T) \times \Gamma_D, \\
      u(0) &= u_0 &&\text{on } \Omega,
    \end{aligned} \qquad \right\}
\end{equation}
and their regularity. We will interpret~\eqref{eq::concrete} as an abstract evolution equation
in a scale $X_\theta$ of function spaces and work within a maximal parabolic regularity
framework. The scale will be $X_\theta \coloneqq [W^{-1,p}_D(\Omega),L^p(\Omega)]_\theta$ with
$p>d$. (All objects and notions will be properly introduced below.)  The defining feature
of~\eqref{eq::concrete} is the coefficient $\xi(u)$ in the divergence operator. The problem
further includes mixed boundary conditions on the disjoint boundary parts $\Gamma_D$ and
$\Gamma_N$ with $\Gamma_D \cup \Gamma_N = \partial\Omega$ where we allow for inhomogeneous
Robin/Neumann data. The setting for $\Omega \subset \R^d$, where $d\in\{2,3\}$, will be that of
a bounded weak Lipschitz domain compatible with the mixed boundary conditions (regular in the
sense of Gr\"oger). The nonlinear functions $\Fkal_\Gamma$ and $\Fkal_\Omega$ have to satisfy a
``local'' Lipschitz condition and we can afford up to linear growth; this is a classical
assumption when aiming for global-in-time results. (We do not assume monotonicity for this
work.)  Finally, we note that there will be no further explicit smoothness assumption on the
coefficient matrix $\mu$, we only assume it to be bounded. We will however require that the
weak divergence operator $-\nabla \cdot \mu\nabla + 1$ admits optimal elliptic regularity in
$W^{-1,p}_D(\Omega)$.

Our main result is that if the data in the abstract formulation
of~\eqref{eq::concrete} yields a well-defined problem in
$X_\theta \coloneqq [W^{-1,p}_D(\Omega),L^p(\Omega)]_\theta$ with
$p>d$, then this problem admits a unique global-in-time solution
$u \in W^{1,s}(0,T;\Dom_{X_\theta}(-\nabla\cdot\mu\nabla+1))\mathrel\cap
L^s(0,T;X_\theta)$ for suitably large $s$. This is
true for all $\theta \in [0,\bar\theta]$ for some
$\bar\theta \in [1-d/p,1]$; thus, we have global-in-time solutions for
a whole scale of function spaces at our disposal which allows for a
flexible treatment of a wide range of applications. The upper bound
$\bar\theta$ depends on the degree of H\"older regularity admitted by
the solutions to the elliptic problem with mixed boundary conditions
associated to $-\nabla\cdot\mu\nabla +1$ in $L^p(\Omega)$. The
reasoning is based on nonautonomous maximal parabolic regularity with
constant domains for
$-\nabla\cdot\eta\mu\nabla$ in $X_\theta$ when the coefficient $\eta>0$
is H\"older continuous in space of order $> \theta$. This is a second
main result. The former property is used to bootstrap the unique
global-in-time solution which exists for $X_0 = W^{-1,p}_D(\Omega)$;
this was established in previous work~(\cite{Meinlschmidt2016}) and is
revisited and improved below.

\emph{Context.} Quasilinear parabolic PDEs arise in several
real-world applications, for instance in the immediate example of heat
conduction when conductivity is tem\-pe\-ra\-ture-dependent, but also
in much more complicated problems such as semiconductor
physics~\cite{Selberherr1984} or liquid crystal
growth~\cite{HieberPruess2016}, see
also~\cite[Introduction]{Amann1995} and the references there. We also
mention the classical monography~\cite{Ladyzhenskaya1968} for early
work in this context. More recently, several
contributions~\cite{HieberRehberg2008,DisserRehberg2019,%
  HieberNesensohnPruessSchade2016,Meinlschmidt2017_2,%
  Meinlschmidt2017_1,Meinlschmidt2016,Bonifacius2018,Dintelmann2009}
have addressed a rather rough geometric setting for such problems with
respect to the domain, boundary conditions, and coefficient functions,
similar to the one listed above.  Such a nonsmooth geometric setting
as described before is often motivated by a realistic and appropriate
model of industrial applications where one has e.g.\ nonsmooth
workpieces made of different materials. The regularity concept of
Gr\"oger has proven to be a sort of umbrella framework in this regard,
see~\cite{Groeger1989,Disser2015} and the references therein. The
approach for~\eqref{eq::concrete} then goes via the concept of
(nonautonomous) \emph{maximal parabolic regularity}, which has turned
out to be quite flexible and useful in this context because it allows
to ``modularize'' the functional-analytic treatment of abstract
evolution equations into several building blocks. We refer to the
fundamental recent works of Amann~\cite{Amann2005_2} and Pr\"uss and
collaborators~\cite{Pruess2002,LeCronePruessWilke2014,KoehnePruessWilke2010}
in this regard. This is particularly useful in the nonsmooth
geometric setting, because many convenient tools from e.g.\ elliptic
regularity theory or (function space) interpolation theory are not
available at all or at least require further justification, and a
modularized ansatz allows to make use of new developments in every
component. In fact, these insights have sparked quite some research
interest in the respective directions and, indeed, there has been
tremendous development here in the recent years, even for (much) more
general geometric settings; we exemplary refer
to~\cite{BrewsterMitreaMitreaMitrea2014,BechtelEgertHaller-Dintelmann2020,Bechtel2019,Disser2015,Meinlschmidt2021,Dintelmann2016}.

However, the question of \emph{global-in-time} solutions to problems
of type~\eqref{eq::concrete} remains ubiquitous, at least as far as it
can be expected from the respective physical model. In the nonsmooth
geometric setup, existence of global-in-time solutions for a less
general instance of the abstract version of~\eqref{eq::concrete} was
established for $X_0=W^{-1,p}_D(\Omega)$ for $p>d$ in~\cite{Meinlschmidt2016}; the
solution is then in $W^{1,s}(0,T;W^{-1,p}_D(\Omega))\cap
L^s(0,T;W^{1,p}_D(\Omega))$ with $\frac1s < \frac12 - \frac{d}{2p}$. The
authors rely on uniform H\"older estimates for nonautonomous linear
parabolic equations established in the same paper and nonautonomous maximal
parabolic regularity; we will revisit and
improve upon the result in Section~\ref{subsec::exonW1p}.
The insight here was that the elliptic differential operator depends
on the coefficient perturbation $\xi(u)$ in a well suited way in the
topology of uniformly continuous functions on $\overline Q$ with
$Q \coloneqq (0,T) \times \Omega$, and this combines very well with
the result on H\"older continuity on $\overline Q$, which is uniform
in the respective coefficient functions, within a Schauder fixed point
argument. Following and extending this train of thought, the authors
of~\cite{Bonifacius2018} were able to show that the equation arising
from inserting the given global-in-time solution in the nonlinear
functions, and re-interpreting as a linear nonautonomous evolution
equation, is well-posed in the Bessel-potential spaces
$X = H^{-\zeta,p}_D(\Omega)$ where $\zeta<1$, but close to $1$. This
is based on the observation that the H\"older continuity of the given
global solution is in fact sufficient to verify nonautonomous maximal
parabolic regularity for the differential operator in this space and
allows to bootstrap the regularity for $u$, if the respective data in
the problem is suitable. In this work, we take this procedure even
further: Since the given global solution is now more regular, its
degree of H\"older continuity, at least in space, has now also
increased, and indeed, we are able to show that this leads to another
bootstrap-improvement of regularity for $u$. This bootstrapping will
work as far as the given data admits. We will explain this in a bit
more detail below. Before, let us mention that there is another recent
result on global-in-time existence in~\cite{Casas2018} to a problem
similar to~\eqref{eq::concrete}, but in a more regular setting, that
is, for pure homogeneous Dirichlet boundary conditions and a
$C^{1,1}$-domain. We comment on the relation of our work
to~\cite{Casas2018} e.g.\ in Remark~\ref{rem::optimal-case} below.

\emph{Motivation.} In addition to a contribution to the framework
of analysis based on maximal regularity techniques for interesting
real-world applications as mentioned before, the original motivation
of the present work comes from optimal control. In fact, several of
the recent results mentioned before were derived in an optimal control
context; this is also true for~\cite{Casas2018}. Indeed, the analysis
of a PDE-constrained optimization problem is usually based on detailed
regularity and stability results for the underlying PDE, in particular
also on global-in-time existence and uniqueness in the first
place. Obtaining the necessary results can be challenging especially
in case of so-called \emph{state constraints}, that is, pointwise
upper and lower bounds on the solution $u$ (the \emph{state}) in every
$(t,x) \in \overline Q$. The difficulty arises here because the amount
of regularity for the state required to deal with this type of
constraint is relatively high; for instance, rigorously establishing
first-order necessary optimality conditions in this context via the
standard technique requires \emph{continuous} solutions of the state
equation, see e.g.~\cite{Casas1993}. For quasilinear parabolic
problems such as~\eqref{eq::concrete} this regularity is guaranteed
within the setting of~\cite{Meinlschmidt2016,Bonifacius2018}, and one
can proceed with the ususal reasoning,
see~\cite{HoppeNeitzel2022,Meinlschmidt2017_2}. In fact, in this
case, continuity of the solution is a crucial aspect of the state
equation analysis in the first place as explained above.
Nevertheless, the derivation of second-order sufficient conditions for
the same type of problem usually requires more regularity for
conceptual reasons, at least for the
linearized state equation; this is a reason why the authors
of~\cite{Casas2008} consider a more regular geometric
setting. Moreover, for even more demanding types of constraints, such
as pointwise bounds on the \emph{gradient} of the state, one requires
even more regularity for the solution to~\eqref{eq::concrete} such as $\nabla u \in C(\overline{Q},\R^d)$  (Such a
constraint would be used e.g.\ to force material stresses to lie
within certain bounds.) In this paper, we thus strive for the optimal
regularity obtainable in dependence on the data within the scale
$X_\theta$ of function spaces.
 
\emph{Outline and organization.} As already mentioned above, this
work contains several new contributions which we collect in the
following milestones: First, we revisit and reprove the global-in-time
existence result for $X_0 = W^{-1,p}_D(\Omega)$
in~\cite[Theorem~5.3]{Meinlschmidt2016} to also incorporate up to
linear growth in the nonlinearity. This is
Theorem~\ref{thm::exonW1p}. Second, we extend the nonautonomous
maximal parabolic regularity result for
$-\nabla\cdot\eta\mu\nabla + 1$ obtained in~\cite{Bonifacius2018} for
$X = H^{-\zeta,p}_D(\Omega)$ to the whole scale $X_\theta$ using
H\"older-continuity of the scalar coefficient perturbation $\eta$ of
degree $> \theta$, see Theorem~\ref{thm::mpronxtheta}. Our proof here
is also much less involved than the one in~\cite{Bonifacius2018} since
we establish an invariance property for the domain of the elliptic
operators with respect to $\eta$. This result then allows to, third,
bootstrap the global-in-time solution $u$ from the ambient space $X_0$
to $X_\theta$ with the associated regularity
$u \in W^{1,s}(0,T;\Dom_{X_\theta}(-\nabla\cdot\mu\nabla+1))\cap
L^s(0,T;X_\theta)$, up to a certain threshold $\bar\theta$. We have
split this result into two parts, Theorems~\ref{thm::eqinXtheta}
and~\ref{thm::regforeq}. If the solutions to elliptic problem with
mixed boundary conditions associated to $-\nabla\cdot\mu\nabla + 1$ in
$L^p(\Omega)$ are at least Lipschitz-continuous, then our procedure
works up to $X_1 = L^p(\Omega)$ indeed. (A particular case would be
that of optimal Sobolev regularity
$W^{2,p}(\Omega)\cap W^{1,p}_D(\Omega)$ for the elliptic problem in
$L^p(\Omega)$.) In this sense, our results close the gap between
existence and uniqueness of global-in-time solutions for the
$X_0 = W^{-1,p}_D(\Omega)$- and the $X_1 = L^p(\Omega)$ setting as
obtained in~\cite{Meinlschmidt2016} and~\cite{Casas2018},
respectively.

To reiterate, our overall strategy to obtain global-in-time solutions
to~\eqref{eq::concrete} in $X_\theta$ can be summarized as follows: We
obtain a global-in-time solution for $\theta = 0$, so in
$X_0 = W^{-1,p}_D(\Omega)$. Re-inserting this solution into the
nonlinear functions, we see that the solution satisfies a linear
nonautonomous problem whose coefficient $\xi(u)$ admits a certain
degree of H\"older-continuity. This degree of H\"older-continuity
$\vartheta$ allows to invoke nonautonomous maximal parabolic
regularity in $X_\theta$ for $\theta < \vartheta$, from which we
obtain that the solution $u$ is, in truth, more regular. But,
repeating the argument, we see that improved regularity of $u$ has
also improved the degree of H\"older-continuity to
$\vartheta^+ > \vartheta$, and we obtain even better regularity for
$u$ than before.  This procedure works iteratively up to some maximal
choice $\bar\theta$ of $\theta$ that ultimately depends on the domain,
the boundary conditions and $\mu$, via the maximal degree of H\"older
continuity for solutions to the elliptic problem with mixed boundary
conditions associated to $-\nabla\cdot\mu\nabla + 1$ in $L^p(\Omega)$.

The paper is organized as follows: In Section~\ref{subsec::subsec21},
we introduce notation and assumptions and collect several auxiliary
concepts and results used throughout the paper. We also formulate
equation~\eqref{eq::concrete} in an appropriate function space
setting, see Section~\ref{sec:probl-stat-assumpt}. In
Section~\ref{subsec::exonW1p}, we provide the global-in-time existence
and uniqueness result for $X_0 = W^{-1,p}_D(\Omega)$ in
Theorem~\ref{thm::exonW1p}. Moreover, we collect the improved
regularity results on $H^{-\zeta,p}_D(\Omega)$
from~\cite{Bonifacius2018} and use this to give a result for
$X = L^{p/2}(\Omega)$, too.  The main step towards the overall main
result is then obtained in Section~\ref{sec::quasilinonXtheta}, where
we establish nonautonomous maximal parabolic regularity for
$-\nabla\cdot\eta\mu\nabla+1$ in $X_\theta$ when $\eta$ is
Hölder-continuous in space of degree $> \theta$; this is
Theorem~\ref{thm::mpronxtheta}. The fundamental achievement here is
that the domains of the foregoing operators in $X_\theta$ coincide
with that of $-\nabla\cdot\mu\nabla + 1$, so, they are invariant under
$\eta$. This insight is based on bilinear interpolation as in
Appendix~\ref{sec::RehbergMeinlschmidt}. Finally, in
Section~\ref{subsec::eqonXtheta} we tackle global-in-time existence
for~\eqref{eq::concrete} in $X_\theta$ by employing the nonautonomous
maximal regularity result from
Section~\ref{sec::quasilinonXtheta}. More precisely, we first collect
some results on interpolation and domains of fractional powers of
$-\nabla\cdot\mu\nabla + 1$ in
Section~\ref{subsec::interpolation-spaces}; these are required for the
internal workings of the bootstrapping machinery. A first version of
the main result, Theorem~\ref{thm::eqinXtheta}, is then provided in
Section~\ref{subsec::simplecase}, to make the argument more
transparent. The second and more comprehensive version is given
afterwards in Section~\ref{subsec::generalcase} via
Theorem~\ref{thm::regforeq}. A discussion of particular cases where we
can determine the upper threshold $\bar\theta$ for $\theta$ more or
less explicitly can be found in Appendix~\ref{sec:hold-regul-ellipt}.

\section{Preliminaries}
\label{subsec::subsec21}

In this section we introduce several notions and definitions. We
further collect several auxiliary results which will be used
throughout the paper. Finally, once the necessary groundwork has been
done, we state the standing assumptions for this work as well as the
abstract quasilinear parabolic PDE problem in more detail.

\subsection{Notation and conventions}

We fix the given time interval $I = (0,T)$ with $T > 0$. All vector
spaces considered are \emph{real} ones. The domain of a closed
operator $A\colon X \to Y$ between Banach spaces $X,Y$ is denoted by
$\Dom_X(A)$; in general, we equip it with the graph norm. If $A$ is
bijective, then $x \mapsto \|Ax\|_Y$ is equivalent to the graph
norm. If $A$ is a closed operator in $X$ and $Z \subseteq X$, then we
denote the domain of the corestriction of $A$ to $Z$ by
$\Dom_Z(A) = \{ x \in X \colon Ax \in Z\}$. Moreover, by $\Lkal(X,Y)$
we refer to the space of bounded linear operators $X \to Y$ with the
operator norm. By $\embeds_c$ and $\embeds_d$ we denote compact and
dense embedding. We use standard notation and definitions for
classical Lebesgue- and H\"older function spaces, also for
Bochner-Lebesgue $L^s(J,X)$ and Bochner-Sobolev $W^{1,s}(J,X)$ spaces
on an interval, and for real-$(\cdot,\cdot)_{\theta,p}$ and complex
$[\,\cdot,\cdot]_\theta$ interpolation spaces. Usually, we will omit
the underlying spatial domain $\Omega$ in the associated function
spaces, if no confusion is likely and no particular point is to be
made. Finally, given an integrability exponent $p \in [1,\infty]$ we
denote by $p' \in [1,\infty]$ the conjugate exponent defined by
$1/p + 1/p' = 1$.

\subsection{Domain and function spaces}

We next introduce the assumptions on the geometry of the underlying
domain $\Omega$ and its boundary parts $\Gamma_N$ and $\Gamma_D$,
followed by some more function spaces for which this geometry is
(partially) important. Here is the standing assumption on the domain:

\begin{assumption}[Geometry]\label{ass::domain-geometry}
  Let $\Omega \subset \R^d$, where $d \in \{2,3\}$, be a bounded
  domain with boundary $\partial \Omega$; let
  $\Gamma_N \subseteq \partial \Omega$ be relatively open, denoting
  the designated \textbf{N}eumann boundary part, and let
  $\Gamma_D = \partial \Omega \setminus \Gamma_N$ denote the
  \textbf{D}irichlet boundary part. We assume that
  $\Omega \cup \Gamma_N$ is regular in the sense of Gröger, that is, a
  weak Lipschitz domain (``Lipschitz manifold'') with a compatibility condition for
  $\overline{\Gamma_N} \cap \Gamma_D$,
  see~\cite{Groeger1989}. Further, we require the additional property
  that every Lipschitz chart in the definition of regular in the sense of
  Gr\"oger can be chosen to be volume-preserving.
\end{assumption}

We will not need the precise technical formulation for the notions in
Assumption~\ref{ass::domain-geometry} which is why we refer
to~\cite{Groeger1989}; see also~\cite{Dintelmann2009_2} for a 3D
characterization. However, a few comments are in order:
\begin{enumerate}[(1)]\item
  Note that we do \emph{not} assume $\partial\Omega$ to (locally) be
  the graph of a Lipschitz function as it is required in the case of a
  \emph{strong} Lipschitz domain,
  cf.~\cite[Definition~1.2.1.1]{Grisvard1985}. Any such  strong
  Lipschitz domain will be a weak Lipschitz domain with
  volume-preserving charts. In particular, it will immediately be regular in the
  sense of Gr\"oger if either $\Gamma_N = \emptyset$ or
  $\Gamma_N = \partial \Omega$,
  see~\cite[Remark~3.3]{Dintelmann2009}. The classical example of a
  weak Lipschitz domain which is not also a strong Lipschitz domain is
  that of a pair of crossing beams in
  3D~\cite[Section~7.3]{Dintelmann2009}.
\item The assumption that the Lipschitz charts be volume-preserving is
  posed \emph{only} in order to be able to utilize results
  from~\cite{Meinlschmidt2016} where this property was used in order
  to avoid technical particularities in a localization/transformation
  procedure. Thus, one could get rid of this particular assumption by
  re-visiting and improving upon the corresponding results
  in~\cite{Meinlschmidt2016}.
\end{enumerate}


\subsubsection*{Function spaces}

We next turn to some more function spaces which require saying a bit
more in their definition and for which the geometry of $\Omega$ is
important. First, for $s \in \R$ and $q \in (1,\infty)$, let
$H^{s,q}(\R^d)$ be the classical Bessel potential spaces. In order to
incorporate a type of zero trace property into the function spaces, we
note that for $s \in (1/q,1+1/q)$, there exists a continuous linear
\emph{trace operator}
$\tr_{D} \colon H^{s,q}(\R^d) \to L^q(\Gamma_D)$,
see~\cite[Theorems~VI.1\&VII.1]{JonssonWallin1984}. Here and in all
what follows, the Lebesgue space on a subset of the boundary
$\partial\Omega$ is equipped with the $(d-1)$-dimensional Hausdorff
measure $\Hkal_{d-1}$. For such $s$, define the Bessel potential
spaces on $\R^d$ incorporating a homogeneous Dirichlet condition on
$\Gamma_D$ by
\[
  H^{s,q}_D(\R^d) \coloneqq  \Bigl\{ f \in H^{s,q}(\R^d)\colon \tr_{D} f
  \equiv 0 \Bigr\}.
\]
Since the trace operator is continuous on $H^{s,q}(\R^d)$, its kernel
$H^{s,q}_D(\R^d)$ is a closed subspace and thus a
Banach space.

The corresponding function spaces on $\Omega$ for $s \geq 0$ and
$s \in (1/q,1+1/q)$, respectively, are now defined by restriction:
\begin{align*}
  H^{s,q}(\Omega) \coloneqq  \Bigl\{ f|_\Omega\colon f \in
  H^{s,q}(\R^d)\Bigr\} \qquad \text{and} \qquad H^{s,q}_D(\Omega)
  \coloneqq  \Bigl\{ f|_\Omega\colon f \in H^{s,q}_D(\R^d)\Bigr\}, 
\end{align*}
and equipped with the canonical quotient norms. This makes them Banach
spaces. We moreover set
$H^{-s,q}(\Omega) \coloneqq (H^{s,q'}(\Omega))^*$ and
$H^{-s,q}_D(\Omega) \coloneqq (H^{s,q'}_D(\Omega))^*$. Then all the
spaces introduced so far are reflexive. Note also that
$H^{0,q}(\Omega) = L^q(\Omega)$. The quotient spaces so far are very
abstract. It will however turn out that they can be related to the
usual, intrinsically defined Sobolev spaces due to
Assumption~\ref{ass::domain-geometry}.

Indeed, for $k \in \N$ and $p \in [1,\infty]$, we define the classical
Sobolev spaces $W^{k,q}(\Omega)$ in the canonical way, that is, the
set of $L^q(\Omega)$ functions whose distributional derivatives up to
order $k$ are regular and represented by $L^q(\Omega)$ functions, with
the $\ell^q$-type norm. Let further
\begin{equation*}
  C_D^\infty(\R^d) \coloneqq \Bigl\{ f \in C_c^\infty(\R^d) \colon
  \dist(\supp f,\Gamma_D)>0 \Bigr\}, \qquad C_D^\infty(\Omega)
  \coloneqq C_D^\infty(\R^d)|_\Omega
\end{equation*}
and set
$W^{1,q}_D(\Omega) \coloneqq
\overline{C_D^\infty(\Omega)}^{W^{1,q}(\Omega)}$ and
$W^{-1,q}_D(\Omega) \coloneqq (W^{1,q'}_D(\Omega))^*$. Then, in our
geometric setting as in Assumption~\ref{ass::domain-geometry}---and,
in fact, also quite far beyond that---, we have
$W^{1,q}_D(\Omega) = H^{1,q}_D(\Omega)$ for $q \in (1,\infty)$, up to
equivalent norms since there is a suitable extension operator for
$W^{1,q}_D(\Omega)$ to
$W^{1,q}(\R^d)$~(\cite[Proposition~B.3]{Bechtel2019}).

By construction, we have the usual Sobolev embeddings within the
$H^{s,q}(\Omega)$ family at hand since these transfer immediately from
$H^{s,q}(\R^d)$ to the quotient type spaces. This includes
Rellich-Kondrachov type compactness results. In particular, they also
hold for $W^{1,q}_D(\Omega)$ by virtue of the last paragraph, and also
for embeddings into the H\"older spaces which we state for further
use.
\begin{lemma}[Sobolev]
  \label{lem::Sobolev-embeddings}
  Let $q \in (1,\infty)$ and $s > \frac{d}q$. Then $ H^{s,q}(\Omega)
  \embeds C^{0,s-\frac{d}q}(\overline\Omega)$. 
\end{lemma}

The following interpolation result based
on main results in~\cite{Bechtel2019} and the duality principle for
complex interpolation~\cite[Corollary~4.5.2]{Bergh1976} will be very
useful. It shows that the Bessel scale on $\Omega$ is indeed an
interpolation scale. We point out that the present
Assumption~\ref{ass::domain-geometry} on the geometry of $\Omega$ implies the
assumptions in~\cite{Bechtel2019}, see the introduction there,
cf.~also~\cite[Appendix~A]{Bonifacius2018}.

\begin{lemma}[{\cite[Theorems~1.1\&1.3]{Bechtel2019}}]
  \label{lem::interpol-Bessel}
  Let $q \in (1,\infty)$ and $\theta \in (0,1)$. Then we have
  \begin{equation*}
    \Bigl[L^q(\Omega),W^{\pm 1,q}_D(\Omega)\Bigr]_\theta =
    \begin{cases}
      H^{\pm \theta,q}_D(\Omega) & \text{if}~\theta > \frac1q, \\[0.5em]
      H^{\pm \theta,q}(\Omega) & \text{if}~\theta<\frac1q
    \end{cases}
  \end{equation*}
  up to equivalent norms.  Moreover,
  \begin{equation*}
    \Bigl[W^{-1,q}_D(\Omega),W^{1,q}_D(\Omega)\Bigr]_{\frac12} = L^q(\Omega)
  \end{equation*}
  again up to equivalent norms.
\end{lemma}

\begin{corollary}
  \label{cor::interpol-cont-embed}
  Let $q>d$ and $\theta \in
  (\frac12+\frac{d}{2q},1]$. Then
  \begin{equation*}
    \bigl[W^{-1,q}_D(\Omega),W^{1,q}_D(\Omega)\bigr]_{\theta} \embeds
    C^{0,2\theta-1-\frac{d}q}(\overline\Omega). 
  \end{equation*}
\end{corollary}

\begin{proof}
  The condition on $\theta$ requires exactly that
  $2\theta-1>\frac{d}q$. Thus, with the reiteration theorem for
  complex interpolation~(\cite[Remark~1.9.3.1]{Triebel1995}) and
  Lemma~\ref{lem::interpol-Bessel}:
  \begin{align*}
    \bigl[W^{-1,q}_D(\Omega),W^{1,q}_D(\Omega)\bigr]_{\theta} &=
    \Bigl[\bigl[W^{-1,q}_D(\Omega),W^{1,q}_D(\Omega)\bigr]_{\frac12},W^{1,q}_D(\Omega)\Bigr]_{2\theta-1}
    \\ &= \bigl[L^q(\Omega),W^{1,q}_D(\Omega)\bigr]_{2\theta-1} =
    H^{2\theta-1,q}_D(\Omega)
  \end{align*}
  and the claim follows by Sobolev embedding (Lemma~\ref{lem::Sobolev-embeddings}).
\end{proof}

Finally, let us note that by Assumption~\ref{ass::domain-geometry},
$\Omega$ is a (weak) Lipschitz domain. Thus, there exists a well
defined trace operator $\tr$ which is continuous 
\begin{equation}
  \tr \colon H^{\theta,q}_D(\Omega) \to L^{q}(\partial\Omega) \qquad
  \text{for}~\theta\in(\tfrac1q,1]\label{eq:trace-operator-continuous}
\end{equation}
and for which $\tr u = u|_{\partial\Omega}$ whenever
$u \in C(\overline\Omega) \cap H^{\theta,q}_D(\Omega)$. This follows for example using a
multiplicative trace inequality as in~\cite[Corollary~1.4.7.1]{Mazya2011} and
Lemma~\ref{lem::interpol-Bessel}, see~\cite[Lemma~3.6]{Dintelmann2009}. The condition
$\theta > \frac1q$ is sharp to have a trace operator. Further, we denote by $\tr_N$ the natural
restriction to $\Gamma_N$, that is, $\tr_N u \coloneqq (\tr u)|_{\Gamma_N}$. Then, by duality,
we immediately have:

\begin{lemma}\label{lem::trace}
  Let $q \in (1,\infty)$ and let $\theta \in [0,\frac1q)$. Then
  \[
    \tr_{N}^*\colon L^q(\Gamma_N) \to
    \bigl[W^{-1,q}_D(\Omega),L^q(\Omega)\bigr]_\theta = H^{\theta-1,q}_D(\Omega)
  \]
  is continuous.
\end{lemma}



\subsection{Maximal parabolic regularity}

We next turn to the notion of maximal parabolic regularity. The case
considered here will be that of \emph{constant domains}.  Given two
Banach spaces $E_1 \embeds_d E_0$, we will use the following
abbreviation for the maximal regularity type spaces:
\[
  \W^{1,r}\bigl(I,(E_0,E_1)\bigr) \coloneqq  W^{1,r}(I,E_0) \cap L^r(I,E_1), \qquad r
  \in (1,\infty).
\]
Suppose that there exists an operator on $E_0$ with domain $E_1$ which
is the generator of an analytic semigroup on $E_0$. (This will always
be satisfied in the following.) Let $A\colon I \to \Lkal(E_1,E_0)$ be
a bounded and measurable operator family such that $A(t)$ is a closed
operator in $E_0$ with domain $E_1$ for each $t \in I$. Then $A$ is
said to satisfy \emph{(nonautonomous) maximal parabolic regularity} on
$L^r(I,E_0)$, if for every $f \in L^r(I,E_0)$ and every
$w_0 \in (E_0,E_1)_{1/r',r}$ there exists a unique solution
$w \in \W^{1,r}(I,(E_0,E_1))$ to the equation
\[
  \partial w + Aw = f \quad \text{ in } L^r(I,E_0), \qquad w(0) = w_0\quad
  \text{ in } (E_0,E_1)_{1/r',r}
\]
where $\partial \colon W^{1,r}(I,E_0) \to L^r(I,E_0)$ denotes the
distributional derivative. (We tacitly identifty $(Aw)(t) = A(t)w(t)$
here.) Equivalently, $A$ satisfies (nonautonomous) maximal parabolic
regularity on $L^r(I,E_0)$ \emph{if and only if} the total
differential operator
\begin{equation}
  \label{eq::max-reg-equiv-total-deriv}
  \bigl(\partial + A,\gamma_0\bigr) \colon  \W^{1,r}\bigl(I,(E_0,E_1)\bigr) \to L^r(I,E_0) \times (E_0,E_1)_{1/r',r} 
\end{equation}
is continuously invertible. Due to the assumption on $E_0$ and $E_1$,
it is also equivalent to consider only the case of initial value
$0$. We refer to e.g.~\cite[Proposition~3.1]{Amann2004}.

If $A$ is in fact autonomous, that is, $A(t) \equiv A$ for every
$t \in I$, then maximal parabolic regularity of $A$ on $L^\rho(I,E_0)$ for
\emph{some} $\rho \in (1,\infty)$ is equivalent to maximal parabolic regularity of
$A$ on $L^r(I,E_0)$ for \emph{any} $r \in (1,\infty)$~(\cite[Remark
6.1d]{Amann2004}). In this case we just say that $A$ satisfies \emph{maximal
  parabolic regularity on $E_0$}.

We will freely use that if $A$ satisfies maximal parabolic regularity
on $E_0$, then so does $A + \lambda$ for any scalar $\lambda$.

The following proposition with a sufficient condition for
nonautonomous maximal parabolic regularity going back
to~\cite{Amann2004,Pruess2001} will be the driving force for our later
considerations:

\begin{proposition}[{\cite{Amann2004,Pruess2001}}]
  \label{prop::pruess-schnaubelt-cont-MPR}
  In the above setting, suppose that
  $A\colon \overline I \to \Lkal(E_1,E_0)$ is continuous and that for
  every $\tau \in \overline I$, the operator $A(\tau)$ satisfies
  maximal parabolic regularity on $E_0$. Then $A$ satisfies
  nonautonomous maximal parabolic regularity on $L^r(I,E_0)$ for every
  $r \in (1,\infty)$.
\end{proposition}

Since an operator satisfying maximal parabolic regularity on $E_0$ is
also the (negative) generator of an analytic semigroup on $E_0$ with domain
$E_1$, the above assumption in this regard is always satisfied in the
context of Proposition~\ref{prop::pruess-schnaubelt-cont-MPR}.

We close this section with embeddings for the maximal regularity
spaces.

\begin{lemma}[{\cite[Theorem~3]{Amann2005}}]
  \label{lem::max-reg-spaces-embed}
  Let $E_1 \embeds_d E_0$ be as above and let $r \in (1,\infty)$. Then
  \begin{equation*}
    \W^{1,r}\bigl(I,(E_0,E_1)\bigr) \embeds C\bigl(\overline I,(E_0,E_1)_{1/r',r}\bigr).
  \end{equation*}
  Further, if $\theta \in \mathopen[0,1-\frac1{r}\mathclose)$ and
  $\tau \in (\frac1r,1-\theta)$, then
    \begin{equation*}
    \W^{1,r}\bigl(I,(E_0,E_1)\bigr) \embeds C^{0,\tau-1/r}\bigl(\overline I,[E_0,E_1]_\theta\bigr),
  \end{equation*}
  and the latter embedding is in fact \emph{compact} when
  $E_1 \embeds_c E_0$.
\end{lemma}

\subsection{The differential operator}

We say that $\rho \colon \Omega \to \R^{d\times d}$ is a
\emph{coefficient function} if it is measurable, bounded, and coercive
in the sense that
\begin{equation*}
  \frac{z \cdot \rho(x) z}{|z|^2} \geq \rho_\bullet > 0 \quad
  \text{for almost all}~x \in \Omega \quad \text{for all}~z
  \in \R^d.
\end{equation*}

If $\rho$ is a coefficient function, then we define the second-order
divergence form operator
\begin{equation*}
  -\nabla\cdot\rho\nabla \colon
W^{1,2}_D(\Omega) \to W^{-1,2}_D(\Omega), \quad \blangle -\nabla\cdot\rho\nabla u,v\brangle \coloneqq \int_\Omega
  \rho\nabla u \cdot \nabla v  \quad (v \in W^{1,2}_D(\Omega)).
\end{equation*}
By the assumptions on the coefficient function $\rho$, the operator
$-\nabla\cdot\rho\nabla$ is continuous. Further, due to the
Lax-Milgram Lemma, $-\nabla \cdot \rho \nabla + 1$ is a (topological)
isomorphism and we can omit the ``$+1$'' if $\Gamma_D \neq
\emptyset$. In the present geometric framework of Assumption~\ref{ass::domain-geometry}, we
then automatically have $\mathcal{H}_{d-1}(\Gamma_D) > 0$. Consider the part of the operator in $W^{-1,q}_D(\Omega)$
for $q>2$. We do not relabel this operator by slight abuse of
notation; it will always be clear from the context which $q$ is
meant. This operator is clearly still bijective, but in general we
will \emph{not} have
$\Dom_{W^{-1,q}_D(\Omega)}(-\nabla\cdot\rho\nabla+1) =
W^{1,q}_D(\Omega)$; at least not for $q$ which are not very close to
$2$~(\cite[Theorem~5.6]{Dintelmann2016}). 

\emph{If}
$\Dom_{W^{-1,q}_D(\Omega)}(-\nabla\cdot\rho\nabla+1) =
W^{1,q}_D(\Omega)$, a situation which we will enforce for a $q>d$ as
one of our main assumptions below, then there are some good
consequences. To set the stage, we first introduce, for any
$q \geq 2$, the part $A$ of
$-\nabla\cdot\rho\nabla$ in $L^q(\Omega)$ by
\begin{align*}
  \Dom (A) &  = \Bigl\{ u \in W^{1,2}_D(\Omega) \cap L^q(\Omega)
  \colon -\nabla\cdot\rho\nabla u \in L^q(\Omega)\Bigr\}, \\[0.25em] A u & =
  -\nabla\cdot\rho\nabla u \qquad (u \in \Dom(A)).
\end{align*}
Again, it will always be clear from context to which $q$ the current
incarnation of $A$ refers. In the general
context of mixed boundary conditions and an irregular domain, the
domain of $A$ will be very difficult to determine.

We next collect a few important properties
of the operators $A$ and $-\nabla\cdot\rho\nabla$. In the next result,
the notion of \emph{positive} operator is meant as
in~\cite[Definition~1.14.1]{Triebel1995}, but it is not fundamental
for the rest of this paper.

\begin{proposition}[{\cite[Proposition~4.6/Theorem~11.5]{AuscherBadrHaller-DintelmannRehberg2014}}]
  \label{prop::divergence-operator-properties-no-iso-yet}
  Let $\rho$ be a coefficient function and let $q \geq 2$. Then we
  have the following:
  \begin{enumerate}[(1)] 
  \item The operators $A$ and $-\nabla\cdot\rho\nabla$ are
    \emph{positive} operators on $L^q(\Omega)$ and
    $W^{-1,q}_D(\Omega)$. In particular, their \emph{fractional powers} are
    well-defined. 
  \item Even more, the operators admit a bounded $\Hkal^\infty$
    calculus. This implies that they exhibit \emph{bounded imaginary
      powers} and \emph{maximal parabolic regularity}.
  \end{enumerate}
  The assertions for $A$ also hold true for $q \in (1,2)$.
\end{proposition}

We can transfer the maximal parabolic regularity property also to the
interpolation spaces between $W^{-1,q}_D(\Omega)$ and
$L^q(\Omega)$. (In fact, the same is true for the bounded imaginary
powers; we use this in Appendix~\ref{app::interpolation}.)

\begin{corollary}[{\cite[Theorem 5.16iv]{Dintelmann2009}}]
  \label{cor::max-reg-interpolation-space}
  Let $\rho$ be a coefficient function and let $q \geq 2$. Then the
  part of $-\nabla\cdot\rho\nabla$ in
  $[W^{-1,q}_D(\Omega),L^q(\Omega)]_\theta$ satisfies maximal
  parabolic regularity for every $\theta \in [0,1]$.
\end{corollary}

Proposition~\ref{prop::divergence-operator-properties-no-iso-yet}
enables us in particular to talk about the \emph{square root} of the
associated operators. For $q >2$, if in fact
\begin{equation}
  \label{eq::topo-iso-cond}
  \Dom_{W^{-1,q}_D(\Omega)}(-\nabla\cdot\rho\nabla+1) =
  W^{1,q}_D(\Omega),
\end{equation}
or, equivalently, $-\nabla\cdot\rho\nabla + 1$ is a topological
isomorphism $W^{1,q}_D(\Omega) \to W^{-1,q}_D(\Omega)$, then for the
square roots we have the following fundamental property at hand, the
\emph{Kato square root property}:

\begin{proposition}[{\cite[Theorem~6.5]{Disser2017_2}}]
  \label{prop::kato-square-root}
  Let $\rho$ be a coefficient function and let $q \geq 2$. Suppose
  that~\eqref{eq::topo-iso-cond} holds true. Then
  $\Dom(A+1)^{1/2} = W^{1,q}_D(\Omega)$, that is,
  \begin{equation*}
    (A+1)^{1/2} \colon W^{1,q}_D(\Omega) \to L^q(\Omega) \quad \text{is a
      topological isomorphism}.
  \end{equation*}
\end{proposition}

The \emph{Kato square root property} classically refers to the case
$q=2$ which is the fundamental basis for the remaining ones; we refer
to the seminal work~\cite{BechtelEgertHaller-Dintelmann2020}, see
also~\cite{AuscherBadrHaller-DintelmannRehberg2014}. It is always
satisfied if $\mu$ is \emph{symmetric}. Note that from
Proposition~\ref{prop::kato-square-root} it also follows that
$(-\nabla\cdot\rho\nabla + 1)^{1/2}$ is a topological isomorphism
$L^q(\Omega) \to W^{-1,q}_D(\Omega)$,
cf.~\cite[Theorem~11.5]{AuscherBadrHaller-DintelmannRehberg2014}.

From the following sketch argument we infer that we can equivalently regard
$(A+1)^{1/2}$ either as the square root of $A+1$ or as the part of the
square root $(-\nabla\cdot\rho\nabla + 1)^{1/2}$ in
$L^q(\Omega)$:
\begin{align*}
  (-\nabla\cdot\rho\nabla + 1)^{-1/2} L^q(\Omega)  &=
  (-\nabla\cdot\rho\nabla + 1)^{-1}(-\nabla\cdot\rho\nabla +
  1)^{1/2} L^q(\Omega) \\ & = (-\nabla\cdot\rho\nabla +
  1)^{-1}W^{-1,q}_D(\Omega) \\ & = W^{1,q}_D(\Omega) = (A+1)^{-1/2}L^q(\Omega).
\end{align*}
With similar reasoning based on the foregoing, we in fact obtain the
same property for all fractional powers:
\begin{equation*}
  \Dom_{L^q(\Omega)}(-\nabla\cdot\rho\nabla+1)^\gamma = \Dom (A+1)^\gamma \qquad (\gamma
  \in [0,1]).
\end{equation*}
This will essentially allow us to get rid of the operator $A$ in the
following considerations, which eases notation significantly. Finally,
we will make free use of the \emph{reiteration theorem} for fractional
powers~(\cite[Theorem~1.15.3]{Triebel1995}), owing to $A$ being
positive and admitting bounded imaginary powers as established in
Proposition~\ref{prop::divergence-operator-properties-no-iso-yet}:
\begin{lemma}
  \label{lem::reiteration-frac-powers}
  Let $\rho$ be a coefficient function and let $q \geq 2$. Suppose
  that~\eqref{eq::topo-iso-cond} holds true. If
  $\gamma = (1-\theta)\alpha + \theta\beta$ for
  $\alpha,\beta \in [0,1]$ and $\theta \in (0,1)$, then
  \begin{equation*}
    \Dom_{L^q}(-\nabla\cdot\rho\nabla+1)^\gamma =
    \Bigl[\Dom_{L^q}(-\nabla\cdot\rho\nabla+1)^\alpha,
    \Dom_{L^q}(-\nabla\cdot\rho\nabla+1)^\beta\Bigr]_\theta.    
  \end{equation*}
\end{lemma}

We gather a permanence principle for the optimal elliptic
regularity property~\eqref{eq::topo-iso-cond}. The upper bound
$q \leq 6$ in the statement is a technical limitation related to the
Sobolev exponent $2^* = 6$ in $d=3$.

\begin{lemma}[{\cite[Lemma~6.2]{Disser2015}}]
  \label{lem::iso-invariance-uniform-cont}
  Let $\rho$ be a coefficient function and let $q \in [2,6]$. Suppose
  that~\eqref{eq::topo-iso-cond} holds true. Then
  \begin{equation*}
    \Dom_{W^{-1,q}_D(\Omega)}(-\nabla\cdot\eta\rho\nabla + 1) = W^{1,q}_D(\Omega)
  \end{equation*}
  for every uniformly continuous positive function
  $\eta \in C(\overline\Omega)$.
\end{lemma}

Finally, for $\varrho \in L^\infty(\Gamma_N)$, we will use
$\Bkal_\varrho\coloneqq \tr_N^*\circ [\varrho\tr_N]$ to signify the
mapping associated to a Robin boundary condition. Indeed, for
$\theta \in (1-\frac1q,1]$, the following operator is well-defined and
continuous due to~\eqref{eq:trace-operator-continuous} and
Lemma~\ref{lem::trace}:
\begin{equation*}
  \blangle\Bkal_\varrho u,v\brangle \coloneqq \int_{\Gamma_N}
  \varrho \, (\tr_N u) \, (\tr_N v), \qquad \Bkal_\varrho \colon
  W^{1,q}_D(\Omega) + C(\overline\Omega) \to H^{-\theta,q}_D(\Omega).
\end{equation*}


\subsection{Problem statement and assumptions}\label{sec:probl-stat-assumpt}

Next, we state the minimal assumptions on the data
of~\eqref{eq::concrete} that will allow us to obtain global-in-time
solutions of~\eqref{eq::concrete} in the
$X_0 = W^{-1,p}_D$-setting. These global-in-time solutions will
be the starting point for the bootstrap procedure through the scale
$X_\theta = [L^q(\Omega),W^{-1,q}_D(\Omega)]_\theta$ in
Section~\ref{sec::quasilinonXtheta}. Of course, it will be necessary
to assume more for some objects---more precisely, for $s,\Fkal$ and
$u_0$---to improve regularity later on.

\begin{assumption}\label{ass::problem-data}
  Suppose that the following properties hold true:
  \begin{enumerate}
  \item[(Co)] The function $\xi\colon \R \to \R$ is locally Lipschitz
    continuous and satisfies
    $0 < \xi_{\bullet} \leq \xi(z) \leq \xi^{\bullet}$ for all
    $z \in \R$. Further, $\mu$ is a coefficient function. For the
    coefficient in the Robin boundary condition we suppose
    $\alpha \in L^\infty(I,L^\infty(\Gamma_N))$.
  \item[(Iso)] We assume that there is $p \in (d,d+1]$ such that
    \begin{align*}
      -\nabla \cdot \mu \nabla  + 1\colon W^{1,p}_D(\Omega) \to W^{-1,p}_D(\Omega)
    \end{align*}
    is a topological isomorphism and \textbf{fix this choice} of
    $p$. Let further
    $\frac{1}{s} < \frac{1}{2}-
    \frac{d}{2p}$.
  \item[($\Fkal$)]
    $\Fkal \coloneqq \Fkal_\Omega + \tr_N^*\Fkal_\Gamma \colon I
    \times (W^{-1,p}_D,W^{1,p}_D)_{1/s',s} \to W^{-1,p}_D$ is a
    locally Lipschitz Caratheodory map, that is, it is measurable with
    respect to the first variable, continuous with respect to the
    second, and for every $R > 0$ there is $L_R \in L^s(I)$ such that
    for almost all $t \in I$:
    \[
      \bigl\lVert \Fkal(t,w_1) - \Fkal(t,w_2) \bigr\rVert_{W^{-1,p}_D}
      \leq L_R(t) \lVert w_1 - w_2
      \rVert_{(W^{-1,p}_D,W^{1,p}_D)_{1/s',s}},
    \]
    whenever $w_1, w_2 \in (W^{-1,p}_D,W^{1,p}_D)_{1/s',s}$ and
    $\lVert w_i \rVert_{(W^{-1,p}_D,W^{1,p}_D)_{1/s',s}} \leq R$,
    $i=1,2$.

    Moreover, $\Fkal$ obeys an up-to-linear growth condition in the second
    variable: There is $\psi \in L^s(I)$ such that for all
    $w \in (W^{-1,p}_D,W^{1,p}_D)_{1/s',s}$ and almost all $t \in I$,
    \[
      \bigl\lVert \Fkal(t,w) \bigr\rVert_{W^{-1,p}_D} \leq \psi(t)
      \bigl(1 + \lVert w \rVert_{C(\overline \Omega)}\bigr).
    \]
  \item[(IV)] The initial condition satisfies
    $u_0 \in (W^{-1,p}_D, W^{1,p}_D)_{1/s',s}$.
  \end{enumerate}
\end{assumption}

The setting described in Assumptions~\ref{ass::domain-geometry}
and~\ref{ass::problem-data} is similar to the one considered
in~\cite[Section~5]{Meinlschmidt2016}. The major difference is that we
now allow the function $\Fkal$ to grow up to linearly with respect to
the function variable instead of imposing a quite harsh global
boundedness assumption as done in~\cite{Meinlschmidt2016}. 

Still, Assumptions~\ref{ass::domain-geometry}
and~\ref{ass::problem-data}, in particular~(Iso), impose non-trivial
conditions on the considered setting. A few comments are thus
in order.

\begin{remark}\label{rem::assumptdisc_1}
  \begin{enumerate}[(1)]
  \item If $\Omega$ is a bounded domain with Lipschitz boundary
    (strong Lipschitz domain) and $\Gamma_N = \emptyset$ or
    $\Gamma_N = \partial \Omega$ and $\mu$ is a symmetric and
    uniformly continuous coefficient function, then there will be
    $p>3$ such that~(Iso) in Assumption~\ref{ass::problem-data} is
    satisfied, see~\cite[Theorem 3.12, Remark 3.17]{Elschner2007}.  In
    this sense, Assumptions~\ref{ass::domain-geometry}
    and~\ref{ass::problem-data} cover the classical ``regular''
    setting of strong Lipschitz domains in spatial dimensions $d=2,3$
    with pure Dirichlet or Neumann boundary conditions and a
    symmetric, uniformly continuous coefficient function.  Further,
    from the pioneering work of Gr\"oger~\cite{Groeger1989} we find
    that Assumption~\ref{ass::problem-data}~(Iso) is \emph{always}
    satisfied for some $p > 2$ within our setting. In fact, this is
    also true under much more general assumptions on the geometry of
    the domain, see~\cite{Dintelmann2016}. 
  \item It is well known that in the presence of mixed boundary
    conditions, the isomorphism property in
    Assumption~\ref{ass::problem-data}~(Iso) can only be expected to
    hold for some $p < 4$ in general due to the Shamir
    counterexample~\cite[Introduction]{Shamir1968}. Moreover, if
    Assumption~\ref{ass::problem-data}~(Iso) is valid, then it is also
    valid for all $q \in [2,p]$ due to interpolation and the
    Lax-Milgram Lemma. In this sense, the upper bound of $d+1$ in
    (Iso) should not be considered as critical. We pose it for
    technical reasons to avoid some case
    distinctions. 
    Nevertheless, the
    authors of~\cite{Disser2015} establish a rich zoo of real-world
    constellations such that Assumption~\ref{ass::problem-data}~(Iso)
    is satisfied in dimension $d=3$ within the constraints of the
    other assumptions. 
  \end{enumerate}
\end{remark}

\begin{remark}\label{rem::assumptdisc_2}
  We point out that the growth
    condition in~Assumption~\ref{ass::problem-data}~($\Fkal$) is
    indeed well-defined since
    \begin{equation*}
      \bigl(W^{-1,p}_D,W^{1,p}_D\bigr)_{1/s',s} \embeds C(\overline\Omega).
    \end{equation*}
    In fact, choose $\frac12+\frac{d}{2p} < \theta < 1-\frac1s$. Then
    $(W^{-1,p}_D,W^{1,p}_D)_{1/s',s} \embeds
    [W^{-1,p}_D,W^{1,p}_D]_{\theta}$ and
    Corollary~\ref{cor::interpol-cont-embed} strikes. Such a growth
    condition is the standard requirement in the analysis of abstract
    semilinear equations, see for
    example~\cite[Corollary~3.3.5]{Henry1981}.  

    Besides rather obvious choices for $u\mapsto\Fkal(u)$ such as a
    ``constant'' function  $f \in L^s(I,W^{-1,p}_D)$ and nonlinear
    Nemytskii operators induced by suitable real functions, we point
    to two particular possible incarnations of $\Fkal$:
    \begin{itemize}
    \item Fix $g \in L^s(I,W^{1,p}_D)$---possibly coming from some
      other differential equation---and set
      $\Fkal(t,w) \coloneqq \nabla \cdot w\mu \nabla g(t)$.  Such a
      drift-type term arises e.g.\ in the modelling of semiconductors and
      has been considered within a semilinear parabolic PDE
      in~\cite{Meinlschmidt2021}.
    \item We can also consider nonlocal-in-space interactions in $\Fkal$ such as
      for example
      $ \Fkal(t,w) \coloneqq \sigma\left( \int_\Omega k(t)w \dd x\right)$,
      where $\sigma\colon \R \to \R$ is bounded and Lipschitz and $k$
      is a suitable kernel with $k(t) \in L^1(\Omega)$.
    \end{itemize}
\end{remark}

\begin{remark}\label{rem::assumptdisc_3}
  The assumption that $\xi$ is uniformly bounded from below is necessary to ensure uniform
  ellipticity of the quasilinear differential operator and cannot be avoided easily. The upper
  bound on $\xi$, however, is not strictly necessary and can be removed utilizing a classical
  Stampacchia argument, cf.~\cite[Theorem~2.1]{Casas2018}. Such an argument does not rely on
  the stronger regularity assumptions posed in the cited work. Similarly, it is also possible
  to discuss nonlinear functions $\Fkal$ that are monotone with respect to the second variable
  but not necessarily of linear growth such as the classical $\Fkal(t,w) = w^3$. However, in
  order to keep the discussion more transparent we decided not to include these technical
  modifications.
\end{remark}

\subsubsection*{Abstract problem formulation}

We next give an abstract but precise formulation
of~\eqref{eq::concrete}. The goal is to find a \emph{global-in-time} solution
$u \in \W^{1,s}(I,(W^{-1,p}_D,W^{1,p}_D))$ to
\begin{align}\label{eq::main-problem-abstract}
  \left.\begin{aligned}
      \partial u - \nabla \cdot \xi(u) \mu \nabla u + u + \Bkal_\alpha u &= \Fkal(u) && \text{in } L^s(I,W^{-1,p}_D),\\
      u(0) &= u_0 &&\text{in } (W^{-1,p}_D, W^{1,p}_D)_{1/s',s} \\ 
    \end{aligned} \quad \right\}
\end{align}
and to give sharp sufficient conditions along the scale
$[W^{-1,p}_D,L^p]_\theta$ for when the solution is in fact more
regular. We start in the $W^{-1,p}_D$-setting ($\theta=0$) because we
can in fact prove existence and uniqueness of a global-in-time
solution there, basing on uniform H\"older estimates for nonautonomous
parabolic evolution equations established in~\cite{Meinlschmidt2016}.

Herein, the Neumann/Robin boundary conditions of~\eqref{eq::concrete}
have been absorbed into the distributional right hand side
of~\eqref{eq::main-problem-abstract} and $\Bkal_\alpha u$,
respectively, cf.\ also
Assumption~\ref{ass::problem-data}~($\Fkal$). The Dirichlet boundary
conditions are prescribed by the underlying function space
$W^{1,p}_D$. The above formulation is self-consistent in
$L^s(I,W^{-1,p}_D)$ for $u \in \W^{1,s}(I,(W^{-1,p}_D,W^{1,p}_D))$ due
to the assumptions on the data in Assumption~\ref{ass::problem-data},
cf.\ also Lemma~\ref{lem::max-reg-spaces-embed} and
Corollary~\ref{cor::interpol-cont-embed}.

As we ultimately plan to establish better regularity for $u$ than
claimed above, we introduce the following convention: Let subspaces
$X\embeds W^{-1,p}_D$ and $Y\embeds W^{1,p}_D$ with $Y \embeds_d X$
and some $r \in (1,\infty)$ be given. Suppose that the solution $u$
to~\eqref{eq::main-problem-abstract} in fact satisfies
$u \in \W^{1,r}(I,(X,Y))$, and that~\eqref{eq::main-problem-abstract}
holds true in $L^r(I,X) \times (X,Y)_{1/r',r}$. Then we say that
\emph{$u$ solves~\eqref{eq::main-problem-abstract} (also) on
  $X$}. Note that this solution will then necessarily unique, since the
one for $X=W^{-1,p}_D$ will be.

\begin{remark}
  \label{rem::inhomo-neumann-gap}
  \begin{enumerate}[(1)]
  \item In view of Lemma~\ref{lem::trace} on the adjoint trace
    operator and the scale $X_\theta = [W^{-1,p}_D,L^p]_\theta$, there
    is a natural threshold at $\theta = 1/p$ above which $X_\theta$
    cannot accommodate any more distributional objects such as
    $\tr_N^* g$ arising from inhomogeneous Neumann/Robin boundary data
    $g$. This also applies to $\Fkal_\Gamma$.
    
  \item In fact, we will also allow for a Robin boundary condition in
    the abstract equation~\eqref{eq::main-problem-abstract} in
    $[W^{-1,p}_D,L^p]_\theta$ \emph{only} for $\theta<1/p$. There, the
    boundary condition is enforced by adding the appropriate term
    $\Bkal_\alpha$ and we can view $\Bkal_\alpha$ as a perturbation of
    the main part of the differential operator
    $-\nabla\cdot\mu\nabla + 1$. This clearly does not work any more
    for $\theta > 1/p$ since in this (stronger) setting, the boundary
    condition is built into the differential operator in a strong
    sense. (This is a feature, not a bug.) However, it would be most
    desirable to also be able to incorporate a---then:
    homogeneous---Robin condition there. The problem is that we do
    lack an analogous result regarding the Kato square root property
    as in Proposition~\ref{prop::kato-square-root} for this case with
    $-\nabla\cdot\mu\nabla + 1 + \Bkal_\alpha$. Since our later
    considerations, in particular in
    Section~\ref{subsec::interpolation-spaces}, are strongly based on
    Proposition~\ref{prop::kato-square-root}, we are unable to
    accommodate $\alpha \neq 0$ at the moment. Any improvement in the
    square root property for the operator including $\Bkal_\alpha$
    would transfer to the present setting immediately.
  \end{enumerate}
\end{remark}

\section{Global-in-time solutions and regularity}
\label{sec::sec2}

In this section we provide existence and uniqueness of global-in-time
solutions in the $W^{-1,p}_D$-$W^{1,p}$-setting
for~\eqref{eq::main-problem-abstract}. This is the fundamental result
on which our further considerations are based on since it delivers the
global solution whose regularity we then can bootstrap. We further
augment this result by briefly reviewing existing results
on improved regularity in $X = H^{-\zeta,p}_D(\Omega)$ for $\zeta$
close to $1$ and in $X = L^{p/2}(\Omega)$,

\subsection{Existence of solutions}
\label{subsec::exonW1p}
It follows our first main result that
extends~\cite[Theorem~5.3]{Meinlschmidt2016}. The latter was proven
using a ``global'' Schauder fixed-point argument based on the other
main result in the paper, uniform Hölder
estimates~\cite[Theorem~2.13]{Meinlschmidt2016}, thereby requiring a
rather strong global boundedness property for $\Fkal$. Here, we rely
on the same uniform Hölder estimates, but rather use them to disprove
finite-time blowup of a local-in-time solution for which we can
tolerate up to linear growth in $\Fkal$.

In all what follows, we take
Assumptions~\ref{ass::domain-geometry} and~\ref{ass::problem-data} for
granted.

\begin{theorem}\label{thm::exonW1p}
  There is a unique global-in-time solution
  $u \in \W^{1,s}(I,(W^{-1,p}_D, W^{1,p}_D))$
  to~\eqref{eq::main-problem-abstract}.
\end{theorem}

Before we start with the proof, let us note that there is $\beta>0$
such that
\begin{equation}
  \W^{1,s}(I,(W^{-1,p}_D, W^{1,p}_D)) \embeds_c
  C^{0,\beta}(\overline I,C^{0,\beta}(\overline\Omega)),\label{eq:maxreg-is-hoelder}
\end{equation}
see Lemma~\ref{lem::max-reg-spaces-embed} and
Corollary~\ref{cor::interpol-cont-embed}. 
 
\begin{proof} We argue in a quite concise way how to obtain a
  local-in-time solution. Similar reasoning and
  associated arguments can be found
  in~\cite[Section~5]{Meinlschmidt2016}
  or~\cite{Meinlschmidt2017_1,Dintelmann2009}.  From
  Corollary~\ref{cor::interpol-cont-embed}, we have
  $(W^{-1,p}_D,W^{1,p}_D)_{1/s',s} \embeds C(\overline\Omega)$. This
  has several consequences: For every $w$ in the interpolation space,
  the operator $-\nabla\cdot\xi(w)\mu\nabla + 1$ satisfies maximal
  parabolic regularity in $W^{-1,p}_D$ with domain $W^{1,p}_D$ due to
  Proposition~\ref{prop::divergence-operator-properties-no-iso-yet}
  and Lemma~\ref{lem::iso-invariance-uniform-cont}. Further,
  $w \mapsto -\nabla\cdot\xi(w)\mu\nabla + 1$ is Lipschitz continuous
  on bounded sets in $(W^{-1,p}_D,W^{1,p}_D)_{1/s',s}$ with values in
  $\Lkal(W^{1,p}_D,W^{-1,p}_D)$, and $\Bkal_{\alpha(t)}$
  gives rise to a continuous linear operator
  $(W^{-1,p}_D,W^{1,p}_D)_{1/s',s} \to W^{-1,p}_D$ for almost every
  $t \in I$, recall Lemma~\ref{lem::trace}.

  In~\eqref{eq::main-problem-abstract}, we transfer the Robin operator
  $\Bkal_\alpha$ to the right-hand side
  $\overline{\Fkal}(u) \coloneqq \Fkal(u) - \Bkal_\alpha u$. From the foregoing considerations we infer that the
  assumptions of the seminal theorem of
  Pr\"uss~\cite[Theorem~3.1]{Pruess2002} are satisfied for the
  resulting equation, such that~\eqref{eq::main-problem-abstract}
  admits a unique maximal local-in-time solution $u$ in the maximal
  regularity class. More precisely, there exists $T^\bullet \in (0,T)$
  such that equation~\eqref{eq::main-problem-abstract} admits a unique
  solution $u$ on $(0,T^\bullet)$ and for every
  $T_\bullet \in (0,T^\bullet)$ we have
  $u \in \W^{1,s}(0,T_\bullet,(W^{-1,p}_D, W^{1,p}_D))$. Thereby,
  $T^\bullet$ is characterized by the property that
  \begin{equation}
    \label{eq:blowup-charact}
    \lim_{t\nearrow T^\bullet}
    u(t) \quad
    \text{does not exist in}~(W^{-1,p}_D,W^{1,p}_D)_{1/s',s}.
  \end{equation}
  We show that in fact $T^\bullet = T$ with
  $u \in \W^{1,s}(0,T,(W^{-1,p}_D, W^{1,p}_D))$, using the linear
  growth assumption on $\Fkal$ posed in
  Assumption~\ref{ass::problem-data}~($\Fkal$) to
  disprove~\eqref{eq:blowup-charact}

  To this end, let $0 \leq \tau < T^\bullet$ be arbitrary for now---to
  be fixed later---and let $T_\bullet \in (\tau,T^\bullet)$. Let
  further $\Ckal(\rho_\bullet,\rho^\bullet)$ be the set of all
  measurable nonautonomous coefficient functions
  $\rho \colon I\times\Omega \to \R^{d\times d}$ which are bounded in
  $L^\infty(I\times\Omega;\R^{d\times d})$ by $\rho^\bullet$ and which
  are uniformly coercive almost everywhere on $I \times \Omega$ with
  coercivity constant $\rho_\bullet$. Then, by the main result
  in~\cite[Theorem~2.13]{Meinlschmidt2016}, for some $\gamma > 0$,
  w.l.o.g. $\gamma < \beta$ with $\beta$ as
  in~\eqref{eq:maxreg-is-hoelder}, the number
  \begin{equation*}
    C_{\partial+A}(\tau,T_\bullet)
    \coloneqq \sup_{\rho \in \Ckal(\xi_\bullet\mu_\bullet,\xi^\bullet\mu^\bullet)}\bigl\|(\partial - \nabla\cdot\rho\nabla +
    1)^{-1}\bigr\|_{\Lkal(L^s(\tau,T_\bullet;W^{-1,p}_D);C^{0,\gamma}_0([\tau,T_\bullet]
      \times \overline\Omega))},
  \end{equation*}
  is well-defined and finite. Hereby, we denote by
  $(\partial - \nabla \cdot \rho \nabla + 1)^{-1}$ the solution
  operator $f \mapsto y$ of the nonautonomous problem
  \begin{align}\label{eq::lineq}
    \left.\begin{aligned}
        \partial y - \nabla \cdot \rho \nabla y + y &= f &&\text{in }
        L^s(I,W^{-1,p}_D), \\ 
        y(\tau) &= 0  &&\text{in } (W^{-1,p}_D,W^{1,p}_D)_{1/s',s}, 
      \end{aligned} \quad \right\}
  \end{align}
  and the index $0$ refers to zero initial condition at time $\tau$
  for functions in
  $C^{0,\gamma}_0([\tau,T_\bullet] \times \overline\Omega)$. It is
  easy to see that we must have
  \begin{equation*}
    C_{\partial+A}(\tau,T_\bullet) \leq C_{\partial+A}(0,T) =: C_{\partial+A}.
  \end{equation*}
  Indeed, every problem instance of~\eqref{eq::lineq} on
  $(\tau,T_\bullet)$ can be embedded into a problem instance
  of~\eqref{eq::lineq} on $(0,T)$ with the same input data size by
  extending and shifting the objects on $(\tau,T_\bullet)$ by zero.
  Then the definition of the operator norm gives the estimate. This
  equicontinuity of the parabolic solution operator in the coercivity-
  and boundedness constants of the coefficient function will be the
  crucial element of the proof. Essentially, it will allow us to
  reason as in the semilinear case.

  We need another global estimate, this time for the Robin operator
  whose mapping properties for almost every $t \in I$ were already
  mentioned above. Since $\alpha$ is assumed to be essentially bounded
  in time, we find (recall Lemma~\ref{lem::trace})
  \begin{multline*}
    \bigl\lVert\Bkal_\alpha\bigr\rVert_{\Lkal(C^{0,\gamma}([\tau,T_\bullet] \times
      \overline\Omega),L^s(\tau,T_\bullet;W^{-1,p}_D))} \\ \leq
    |T_\bullet-\tau|^{\frac1s} \lVert
    \tr\rVert_{\Lkal(W^{1,p'}_D,L^{p'}(\Gamma_N))}
    \lVert\alpha\rVert_{L^\infty(I,L^\infty(\Gamma_N))} \eqqcolon
    C_\alpha(\tau,T_\bullet).
  \end{multline*}

  Now, in order to use the foregoing estimate with $C_{\partial + A}$,
  we split off the ``initial value'' $u(\tau)$ of the maximal solution
  $u$ starting from $\tau$ via
  $w(t) \coloneqq e^{-(t-\tau)\Delta}u(\tau)$. We have
  $w \in \W^{1,s}(\tau,T,(W^{-1,p}_D, W^{1,p}_D))$
  by~\cite[Proposition~III.4.10.2]{Amann1995}. Set $v \coloneqq u-w$
  to obtain $v(\tau) = 0$. Then $v  \in
  \W^{1,s}(\tau,T_\bullet,(W^{-1,p}_D, W^{1,p}_D))$ satisfies 
  \begin{equation*}
    \partial v - \nabla \cdot\xi(u)\mu\nabla v + v = \Fkal(u) -
    \Bkal_\alpha u -
    \partial w + \nabla \cdot\xi(u)\mu\nabla w - w, \quad v(\tau) = 0
  \end{equation*}
  in $L^s(\tau,T_\bullet;W^{-1,p}_D)$. So, using the definition of
  $C_{\partial + A}$, the foregoing estimates, the
  embedding~\eqref{eq:maxreg-is-hoelder}, and the linear growth
  assumption for $\Fkal$ as in
  Assumption~\ref{ass::problem-data}~($\Fkal$), we obtain:
  \begin{multline*}
    \|v\|_{C_0^{0,\gamma}([\tau,T_\bullet]\times \overline\Omega)}\\
    \leq C_{\partial +
      A}\left[\bigl(\|\psi\|_{L^s(\tau,T_\bullet)}+C_{\alpha}(\tau,T_\bullet)\bigr)
      \bigl(\|v\|_{C^{0,\gamma}([\tau,T_\bullet] \times
        \overline\Omega)}+\|w\|_{C^{0,\gamma}([\tau,T_\bullet]
        \times \overline\Omega)}\bigr)\right. \\
    \left. +~\bigl(1+\xi^\bullet\mu^\bullet\bigr) \|w\|_{
        \W^{1,s}(\tau,T_\bullet,(W^{-1,p}_D, W^{1,p}_D))} \right].
  \end{multline*}
  We now choose $\tau$ to be
  \begin{equation*}
    \tau \coloneqq \inf\Bigl\{ s \in (0,T^\bullet) \colon
    C_{\partial+A}\bigl(\|\psi\|_{L^s(s,T^\bullet)}+C_{\alpha}(s,T^\bullet)\bigr)
    \leq \frac12
    \Bigr\}. 
  \end{equation*}
  Note that $\tau < T^\bullet$. In particular, with the chosen $\tau$,
  we have
  \begin{equation*}
    C_{\partial+A}\bigl(\|\psi\|_{L^s(\tau,T_\bullet)}+C_{\alpha}(\tau,T_\bullet)\bigr)
    \leq \frac12
  \end{equation*}
  for \emph{all} $T_\bullet \in (\tau,T^\bullet]$. Thus we are able to
  absorb
  $\|v\|_{C^{0,\gamma}_0([\tau,T_\bullet] \times \overline\Omega)}$ on
  the left in the last estimate. Hence, for any
  $T_\bullet \in (\tau,T^\bullet)$
  \begin{multline*}
    \|v\|_{C^\gamma_0([\tau,T_\bullet] \times \overline\Omega)} \leq
    2C_{\partial+A}
    \left[\bigl(\|\psi\|_{L^s(0,T)}+C_{\alpha}(0,T)\bigr)\|w\|_{C^{0,\gamma}([\tau,T]
        \times \overline\Omega)}\right.  \\
    \left. +~\bigl(1+\xi^\bullet\mu^\bullet\bigr) \|w\|_{
        \W^{1,s}(0,T,(W^{-1,p}_D, W^{1,p}_D))} \right].
  \end{multline*}
  The right-hand side is \emph{independent} of $T_\bullet$. Thus,
  denoting the bound by $D$,
  \begin{equation}
    \limsup_{T_\bullet \nearrow T^\bullet}
    \|v\|_{C^{0,\gamma}([\tau,T_\bullet] \times \overline\Omega)} \leq
    D < \infty.\label{eq:global-proof-limsup}
  \end{equation}
  This implies that
  $v \in C([\tau,T^\bullet] \times \overline\Omega)$: Let
  $t_k \nearrow T^\bullet$. Then $(v(t_k))$ is bounded in
  $C^{0,\gamma}(\overline\Omega)$ by~\eqref{eq:global-proof-limsup}
  and due to Arzel\`{a}-Ascoli, there exists a subsequence (with the
  same name) such that $v(t_k) \to V$ in $C(\overline\Omega)$. Set
  $v(T^\bullet) \coloneqq V$. Let $s_k \nearrow T^\bullet$ be
  arbitrary. By~\eqref{eq:global-proof-limsup}, we find
  \begin{align*}
    \lVert v(T^\bullet) -
    v(s_k)\rVert_{C(\overline\Omega)} & \leq \lVert v(T^\bullet) -
    v(t_k)\rVert_{C(\overline\Omega)} + \lVert v(t_k) -
    v(s_k)\rVert_{C(\overline\Omega)} \\ & \leq \lVert v(T^\bullet) -
    v(t_k)\rVert_{C(\overline\Omega)} + D \lvert t_k
    -s_k\rvert^\gamma \quad \longrightarrow \quad 0,
  \end{align*}
  so $v$ is continuous at $T^\bullet$ and
  $v \in C([\tau,T^\bullet] \times \overline\Omega)$.
  
  Now finally consider the equation
  \begin{equation}\label{eq::splitting}
    \partial z - \nabla\cdot\xi(v+w)\mu\nabla z + z +
    \Bkal_\alpha z = \Fkal(v+w),
    \quad z(\tau) = u(\tau).
  \end{equation}
  Note that
  $t\mapsto \partial - \nabla\cdot\xi(v(t)+w(t))\mu\nabla + 1$ is
  continuous on $[\tau,T^\bullet]$ as an operator from the maximal
  regularity space into $L^s(\tau,T^\bullet,W^{-1,p}_D)$. Thus,
  $-\nabla\cdot\xi(v+w)\mu\nabla + 1$ satisfies nonautonomous maximal
  parabolic regularity on $L^s(\tau,T^\bullet,W^{-1,p}_D)$ with the
  domain $W^{1,p}_D$ via
  Propositions~\ref{prop::pruess-schnaubelt-cont-MPR}
  and~\ref{prop::divergence-operator-properties-no-iso-yet} due
  to~\eqref{eq:maxreg-is-hoelder} and the permanence principle in
  Lemma~\ref{lem::iso-invariance-uniform-cont}. It follows
  that~\eqref{eq::splitting} admits a unique solution
  $z \in \W^{1,s}(\tau,T^\bullet,(W^{-1,p}_D, W^{1,p}_D))$. In
  particular, the limit
  $\lim_{t \nearrow T^\bullet} z(t) = z(T^\bullet)$ exists in
  $(W^{-1,p}_D,W^{1,p}_D)_{1/s',s}$ by
  Lemma~\ref{lem::max-reg-spaces-embed}. But by construction $v+w = u$
  and thus $z = u$ on $[\tau,T^\bullet)$, so the blowup
  criterion~\eqref{eq:blowup-charact} must have been false. Thus $u$
  is a global solution.
\end{proof}

Note that in the foregoing proof~\eqref{eq::splitting} is well
defined, although $\Fkal$ is defined only on
$(W^{-1,p}_D, W^{1,p}_D)_{1/s',s}$, instead of $C(\overline \Omega)$
as in~\cite{Meinlschmidt2016}. This is because the right hand side
of~\eqref{eq::splitting} is still well-defined in
$L^s(\tau,T^\bullet;W^{-1,p}_D)$ by the growth condition in
Assumption~\ref{ass::problem-data}~($\Fkal$)
and~\eqref{eq:global-proof-limsup}.

\subsection{First results on improved regularity}\label{sec::first-results-impr}

From the previous existence result, one can hope to bootstrap the
regularity of $u$ provided that the data in the problem is
compatible. We show how such an argument could be made. To this end,
let us first cite the following result, adapted to our setting:

\begin{theorem}[{\cite{Bonifacius2018}, Theorem~3.20}]
  \label{thm::exonH-zetap}
  Let $u$ be the unique solution to~\eqref{eq::main-problem-abstract}
  from Theorem~\ref{thm::exonW1p}.  Let
  $\zeta \in (\frac{d}p,1)$ and
   suppose that $\frac2s < (\zeta-\frac{d}p)\wedge (1-\zeta)$.
  Assume further that
  \begin{equation*}
    u_0 \in (H^{-\zeta,p}_D, \Dom_{H^{-\zeta,p}_D}(-\nabla \cdot \mu
    \nabla+1))_{1/s',s} \quad \text{and that} \quad \Fkal(u) \in L^s(I,H^{-\zeta,p}_D).
  \end{equation*}
 Then in fact
  $u \in \W^{1,s}(I,(H^{-\zeta,p}_D,\Dom_{H^{-\zeta,p}_D}(-\nabla
  \cdot \mu \nabla+1)))$ and $u$ is the unique solution
  to~\eqref{eq::main-problem-abstract} on $H^{-\zeta,p}_D$.
\end{theorem}

Let us briefly recall the
main idea of the proof in~\cite{Bonifacius2018}, starting from
Theorem~\ref{thm::exonW1p}. The core of the reasoning is that the
degree of H\"older continuity of the given solution $u$ suffices to
show that the operator $-\nabla \cdot \xi(u) \mu \nabla + 1$ satisfies
nonautonomous maximal parabolic regularity in $H^{-\zeta,p}_D$ when
$s$ is large enough, depending on $\zeta$. This is established in a
quite involved and technical chain of arguments basing on the
Acquistapace-Terreni condition. From there, one argues quite easily
that the problem
\begin{equation*}
  \partial w - \nabla \cdot \xi(u) \mu \nabla w + w =
  \Fkal(u)-\Bkal_\alpha u \quad \text{in}~L^s(I,H^{-\zeta,p}_D), \qquad w(0) = u_0
\end{equation*}
admits a unique solution
$w \in \W^{1,s}(I,(H^{-\zeta,p}_D,\Dom_{H^{-\zeta,p}_D}(-\nabla \cdot
\mu \nabla+1)))$, which, due to uniqueness of $u$ in the $W^{-1,p}_D$
setting, must coincide with $u$.

We have already seen in Lemma~\ref{lem::interpol-Bessel} that in fact
$H^{-\zeta,p}_D = [W^{-1,p}_D,L^p]_{1-\zeta}$.  In the present paper we
systematically follow the above bootstrapping ansatz to discuss
solutions to~\eqref{eq::main-problem-abstract} in
$[L^p,W^{-1,p}_D]_\theta$ with general $\theta \in [0,1]$. It is a
welcome byproduct that our reasoning will be less involved than the
one in~\cite{Bonifacius2018} and that we can dispose of the
requirement $\frac2s < 1-\theta$.

Let us conclude this section by sketching an associated result on the
interpolation scale $[L^{p/2},W^{-1,p}]_\theta$. Some might see this
as the more natural scale as it is well known that thresholds such as
$p>d$ to obtain spatially continuous solutions apply exactly to the
two spaces there. Indeed, in this case, an involved bootstrapping
argument is not necessary at all, since improved regularity related to
$X=L^{p/2}$ is an immediate consequence of
Theorem~\ref{thm::exonH-zetap} and $u \in C(\overline I,W^{1,p}_D)$:

\begin{theorem}\label{thm::exonlp2}
  Consider the setting of Theorem~\ref{thm::exonH-zetap}. Suppose
  that $\alpha = 0$. Assume additionally that
  \begin{equation*}
    u_0 \in (L^{p/2}, \Dom_{L^{p/2}}(-\nabla \cdot \mu
    \nabla+1))_{1/s',s} \quad \text{and that} \quad \Fkal(u) \in L^s(I,L^{p/2}).
  \end{equation*}
  Then
  $u \in \W^{1,s}(I, (L^{p/2}, \Dom_{L^{p/2}}(-\nabla \cdot \mu
  \nabla+1)))$ and $u$ is the unique solution
  to~\eqref{eq::main-problem-abstract} on $L^{p/2}$.
\end{theorem}

\begin{proof}
  We bootstrap the solution from Theorem~\ref{thm::exonH-zetap} once
  more. (Note that $L^{p/2} \embeds H^{-\zeta,p}_D$ due to the
  assumptions on $\zeta$ and $p$ and (adjoint) Sobolev embedding.) From
  Theorem~\ref{thm::exonH-zetap} and maximal regularity space
  embeddings, we in fact have $u \in C(\overline I,W^{1,p}_D)$,
  see~\cite[Corollary~3.7]{Bonifacius2018}
  or~\cite[Lemma~6.16]{Dintelmann2009}. But
  $\Dom_{L^{p/2}}(-\nabla \cdot \mu \nabla+1)$ is invariant under a
  coefficient perturbation by a $W^{1,p}_D$-function,
  see~\cite[Lemma~6.7/Corollary~6.8]{Dintelmann2009}. Thus
  \begin{equation*}
    \Dom_{L^{p/2}}(-\nabla \cdot \xi(u(t)) \mu
    \nabla+1) = \Dom_{L^{p/2}}(-\nabla \cdot \mu
    \nabla+1) \qquad \text{for every}~t\in \overline I,
  \end{equation*}
  in particular
  \[
    t \mapsto -\nabla \cdot \xi(u(t))\mu \nabla + 1 \in C(\overline I,
    \Lkal\bigl(\Dom_{L^{p/2}}(-\nabla \cdot \mu
    \nabla+1),L^{p/2})\bigr).
  \]
  Since for each fixed $\tau \in \overline I$ the autonomous operator
  $-\nabla \cdot \xi(u(\tau))\mu \nabla + 1$ admits maximal parabolic
  regularity on $L^{p/2}$
  (Proposition~\ref{prop::divergence-operator-properties-no-iso-yet})
  with the common domain $\Dom_{L^{p/2}}(-\nabla \cdot \mu \nabla+1)$,
  we conclude by Proposition~\ref{prop::pruess-schnaubelt-cont-MPR}
  that also $-\nabla \cdot \xi(u)\mu \nabla + 1$ admits
  \emph{nonautonomous} maximal parabolic regularity on $L^{p/2}$. The
  claim follows.
\end{proof}

Let us briefly come back to the motivation outlined in the
Introduction, regarding optimality conditions for an optimal control
problem associated with~\eqref{eq::main-problem-abstract} with
pointwise constraints on the gradient of the state. For these, we want
to guarantee that $\nabla u \in C(\overline{Q},\R^d)$ for a solution $u$
to~\eqref{eq::main-problem-abstract}. But, in general, we will have
$\frac{p}{2} < d$---see Remark~\ref{rem::assumptdisc_1}---, so this
regularity cannot be obtained even in the optimal case for
Theorem~\ref{thm::exonlp2}, which is
$u \in W^{1,s}(I,L^{p/2}) \cap L^s(I,W^{2,p/2}\cap W^{1,p/2}_D)$. In
this sense, even Theorem~\ref{thm::exonlp2} is insufficient for this
endeavor. On the scale $X_\theta = [W^{-1,p},L^p]_\theta$, however, it
\emph{is} possible to obtain the necessary regularity under quite
optimal conditions. This is the main reason not to be content with
Theorem~\ref{thm::exonlp2} but to work on the scale $X_\theta$.

\section{Nonautonomous maximal parabolic regularity on the scale}
\label{sec::quasilinonXtheta}

As outlined before, the main driving force behind the regularity
bootstrapping procedure will be nonautonomous maximal parabolic
regularity for $-\nabla\cdot\xi(u)\mu\nabla + 1$ in better spaces than
$W^{-1,p}_D$. This is the topic of the present section. We formally
introduce the following spaces that will also be used throughout the
rest of the paper:
\[
  X_\theta \coloneqq  [W^{-1,p}_D, L^p]_{\theta}, \qquad Y_\theta \coloneqq 
  \Dom_{X_\theta}(-\nabla \cdot \mu \nabla + 1), \qquad \theta \in
  (0,1).
\]
To simplify notation, we set $X_0 = W^{-1,p}_D$ and $X_1 = L^p$ as
well as $Y_0 = W^{1,p}_D$ and
$Y_1 = \Dom_{L^p}(-\nabla \cdot \mu \nabla + 1)$. Recall from
Lemma~\ref{lem::interpol-Bessel} that for $\theta \neq 1/p$, the space
$X_\theta$ in fact coincides with a (dual) Bessel space which either
reflects the Dirichlet boundary condition ($\theta < 1/p$) or not
($\theta > 1/p$). In particular, we have $Y_\theta \embeds X_\theta$ with dense embedding.

We will be concerned with deriving sharp conditions
on the scalar coefficient perturbation $\eta$ to have maximal
parabolic regularity for $-\nabla\cdot\eta\mu\nabla + 1$ in the
following problem:
\begin{align}\label{eq::nonautomous-general-problem-scale}
  \left. \begin{aligned}
      \partial_t w - \nabla \cdot \eta \mu \nabla w + w &= f &&\text{in } L^r(I,X_\theta), \\
      w(0) &= w_0 &&\text{in } (X_\theta,Y_\theta)_{1/r',r} 
    \end{aligned} \qquad \right\}
\end{align}
Hereby, we expect the coefficient function $\eta\colon Q \to \R$ to
require a certain Hölder-regularity that will be made precise
subsequently. This is an educated guess based on the fact that for
$\theta = 0$, so $X_\theta = W^{-1,p}_D$, we have seen in
Lemma~\ref{lem::iso-invariance-uniform-cont} and
Propositions~\ref{prop::pruess-schnaubelt-cont-MPR}
and~\ref{prop::divergence-operator-properties-no-iso-yet} that $\eta
\in C(\overline Q)$ uniformly continuous is an appropriate choice.

Accordingly, the first and main step is to show an analogue to
Lemma~\ref{lem::iso-invariance-uniform-cont}, that is, that the domain
of the elliptic operator $-\nabla \cdot \eta \mu \nabla + 1$ in
$X_\theta$ is again
$Y_\theta = \Dom_{X_\theta}(-\nabla \cdot \mu \nabla + 1)$ for a
positive function $\eta \in C^{0,\vartheta}(\overline \Omega)$ with
$\vartheta \in (\theta,1]$.

To do so we consider the linear map
\begin{equation}\label{eq::perturbation-operator}
  \eta \mapsto (-\nabla \cdot \eta \mu \nabla ) (-\nabla \cdot \mu \nabla + 1)^{-1}
\end{equation}
in the following limiting settings:
\begin{equation*}
  C(\overline \Omega) \to \Lkal(W^{-1,p}_D) \qquad \text{and}
  \qquad C^{0,1}(\overline \Omega) \to \Lkal(L^p).
\end{equation*}
These maps are well-defined and continuous; for the first one we refer
to Assumption~\ref{ass::problem-data}~(Iso) and H\"older's
inequality. The second is taken care of by the following lemma and
$C^{0,1}(\overline \Omega) \embeds
W^{1,\infty}(\Omega)$~(\cite[Remark~4.2]{Heinonen2005}):
\begin{lemma}
  \label{lem::Lp-domain-embed}
  Let $\eta \in W^{1,\infty}(\Omega)$. Then
    $\Dom_{L^p}(-\nabla\cdot\mu\nabla + 1) \embeds
    \Dom_{L^p}(-\nabla\cdot\eta\mu\nabla)$ and
    \begin{equation*}
    \bigl\lVert (-\nabla \cdot \eta \mu \nabla) (-\nabla \cdot \mu
    \nabla+1)^{-1} \bigr\rVert_{\Lkal(L^p)} \lesssim \lVert \eta
    \rVert_{W^{1,\infty}}.
    \end{equation*}
\end{lemma}
\begin{proof}
  Let $u \in \Dom_{L^p}(-\nabla\cdot\mu\nabla + 1) \embeds W^{1,p}_D$
  (cf.\ Proposition~\ref{prop::kato-square-root}), and let
  $v \in C_D^\infty(\Omega)$. Then
  \begin{equation*}
    \blangle -\nabla\cdot\eta\mu\nabla u,v\brangle \coloneqq
    \int_\Omega \eta\mu\nabla u \cdot\nabla v =
    \int_\Omega \mu\nabla u \cdot \bigl[\nabla(\eta v)-v\nabla \eta\bigr].
  \end{equation*}
  Hence, with $\eta \in W^{1,\infty}$, 
  \begin{multline*}
    \left|\blangle -\nabla\cdot\eta\mu\nabla u,v\brangle\right| \leq
    \|u\|_{\Dom_{L^p(-\nabla\cdot\mu\nabla + 1)}} 
    \|\eta v\|_{L^{p'}}  +
    \|\mu\|_{L^\infty} \|\nabla u\|_{L^{p}} \|v\|_{L^{p'}}\|\nabla
    \eta\|_{L^\infty} \\ \lesssim \|\eta\|_{W^{1,\infty}}\|u\|_{\Dom_{L^p(-\nabla\cdot\mu\nabla + 1)}}\|v\|_{L^{p'}}.    
  \end{multline*}
  Since $C_D^\infty(\Omega)$ is dense in $L^{p'}(\Omega)$, it follows that
  $-\nabla\cdot\eta\mu\nabla u$ defines a continuous linear functional
  on $L^{p'}(\Omega)$. So $u \in \Dom_{L^p}(-\nabla\cdot\eta\mu\nabla)$, and
  the last estimate shows that
  $\Dom_{L^p}(-\nabla\cdot\mu\nabla + 1) \embeds
  \Dom_{L^p}(-\nabla\cdot\eta\mu\nabla)$. Hence~\eqref{eq::perturbation-operator}
  gives rise to a continuous linear operator on $L^p(\Omega)$. The stated norm
  estimate also follows immediately.
\end{proof}

With bilinear interpolation as in Theorem~\ref{thm::interpolationthm},
we then find that the linear operator
in~\eqref{eq::perturbation-operator} is also well-defined and
continuous as an operator
\begin{equation}\label{eq::applyThmA1}
  \bigl[C(\overline \Omega), C^{0,1}(\overline \Omega)\bigr]_\theta
  \to \Lkal\bigl([W^{-1,p}_D, L^p]_\theta\bigr) =
  \Lkal(X_\theta)
\end{equation}
for each $\theta \in [0,1]$.  The following invariance principle,
which is similar to~\cite[Lemmas 6.6, 6.7]{Dintelmann2009}, relies on
this observation; in fact it will serve as cornerstone of our further
analysis.

\begin{lemma}\label{lem::constdomain}
  Let $\theta \in [0,1]$ be given.
  \begin{enumerate}[(1)]
  \item Let
    $\eta \in [C(\overline \Omega), C^{0,1}(\overline
    \Omega)]_\theta$. Then
    $Y_\theta \embeds \Dom_{X_\theta}(-\nabla \cdot \eta\mu \nabla
    +1)$ and the map
    \[
      [C(\overline \Omega), C^{0,1}(\overline
    \Omega)]_\theta \to \Lkal(Y_\theta,
      X_\theta), \qquad \eta \mapsto -\nabla \cdot \eta \mu \nabla
    \]
    is bounded and linear.
  \item Suppose that
    $0 < \eta,\eta^{-1} \in [C(\overline \Omega), C^{0,1}(\overline
    \Omega)]_\theta$.
    Then
    \[
      Y_\theta = \Dom_{X_\theta}(-\nabla \cdot \eta \mu \nabla + 1),
    \]
    with equivalent norms.
  \end{enumerate}
  In particular, the assertions hold true for a positive function
  $\eta \in C^{0,\vartheta}(\overline\Omega)$ with
  $\vartheta \in (\theta,1]$ if $\theta <1$ and $\vartheta = 1$ if
  $\theta = 1$.
\end{lemma}

\begin{proof}
  (1)~For the first assertion, it will be enough to show that
  $Y_\theta \embeds \Dom_{X_\theta}(-\nabla \cdot \eta \mu \nabla)$. For
  $\varphi \in Y_\theta$ we have
  \begin{align*}
    -\nabla \cdot \eta \mu \nabla \varphi = \underbrace{\left[
        (-\nabla \cdot \eta \mu \nabla) (-\nabla \cdot \mu
        \nabla+1)^{-1}\right]}_{\in \Lkal(X_\theta) \text{ due
        to~\eqref{eq::applyThmA1}}} \underbrace{(-\nabla \cdot \mu
      \nabla \varphi+\varphi)}_{\in X_\theta}  \in X_\theta,
  \end{align*}
  so indeed $\varphi \in \Dom_{X_\theta}(-\nabla \cdot \eta \mu \nabla)$.
  We further use that the operator
  in~\eqref{eq::perturbation-operator} is continuous as
  in~\eqref{eq::applyThmA1} with its norm bounded by
   \[
    \bigl\lVert (-\nabla \cdot \eta \mu \nabla) (-\nabla \cdot \mu
    \nabla+1)^{-1} \bigr\rVert_{\Lkal(X_\theta)} \lesssim \lVert \eta
    \rVert_{[C(\overline \Omega), C^{0,1}(\overline \Omega)]_\theta}.
  \]
  to estimate
  \begin{align*}
     \lVert -\nabla \cdot \eta \mu \nabla \varphi \rVert_{X_\theta}
    & \leq \bigl\lVert (-\nabla \cdot \eta \mu \nabla)(-\nabla \cdot
    \mu \nabla+1)^{-1} \bigr\rVert_{\Lkal(X_\theta)} \lVert -\nabla
    \cdot \mu \nabla \varphi + \varphi\rVert_{X_\theta} \\ 
    & \lesssim \lVert \eta \rVert_{[C(\overline \Omega), C^{0,1}(\overline \Omega)]_\theta} \lVert \varphi \rVert_{Y_\theta}.
  \end{align*}
  The assertions follow. 

  (2) 
  Let
  $\eta,\eta^{-1} \in [C(\overline \Omega), C^{0,1}(\overline
  \Omega)]_\theta$. Then both $\eta$ and $\eta^{-1}$ are uniformly
  continuous, bounded and positive. With
  Lemma~\ref{lem::iso-invariance-uniform-cont} it follows that
  $-\nabla\cdot\eta\mu\nabla + 1$ is also a topological isomorphism
  between $W^{1,p}_D$ and $W^{-1,p}_D$, i.e., $\hat \mu \coloneqq \eta \mu$
  satisfies the isomorphism assumption in
  Assumption~\ref{ass::problem-data}~(Iso). Consequently, we can apply the arguments
  of the first part to $\hat \eta \coloneqq \eta^{-1}$ and $\hat \mu$
  instead of $\eta$ and $\mu$, respectively, to obtain that
  \begin{multline*}
    \Dom_{X_\theta}(-\nabla \cdot \eta \mu \nabla + 1) =
    \Dom_{X_\theta}(-\nabla \cdot \hat \mu \nabla + 1) \\
    \embeds \Dom_{X_\theta}(-\nabla \cdot \hat \eta \hat \mu
    \nabla + 1) = \Dom_{X_\theta}(-\nabla \cdot \mu \nabla+1),
  \end{multline*}
  and this implies the claim.

  Finally, by Proposition~\ref{prop::interpolationhoelder} we have
  $C^{0,\vartheta}(\overline \Omega) \embeds [C(\overline \Omega),
  C^{0,1}(\overline \Omega)]_\theta$ when $\vartheta \in
  (\theta,1]$. Moroever, if $\eta \in C^{0,\vartheta}(\overline
  \Omega)$ is positive, then also $\eta^{-1} \in C^{0,\vartheta}(\overline
  \Omega)$. Hence, for such $\eta$ all the previous assertions are
  valid. The case $\theta = 1$ follows analogously.
\end{proof}

Next, we establish nonautonomous maximal parabolic regularity of the
operator $-\nabla \cdot \eta \mu \nabla + 1$ on $X_\theta$ for
positive coefficient functions
$\eta \in C(\overline I,C^{0,\vartheta}(\overline\Omega))$ for
$\vartheta \in (0,1]$. This will give existence, uniqueness and
regularity of solutions
to~\eqref{eq::nonautomous-general-problem-scale}. Since for
$\eta \in C(\overline I,C^{0,\vartheta}(\overline\Omega))$ for
$\vartheta \in (0,1]$ we already have constant domains for
$-\nabla \cdot \eta \mu \nabla + 1$ in $X_\theta$ at hand by virtue of
Lemma~\ref{lem::constdomain}, the following theorem follows quite
easily with Proposition~\ref{prop::pruess-schnaubelt-cont-MPR}:

\begin{theorem}\label{thm::mpronxtheta}
  Given $\vartheta \in (0,1]$ and $\theta \in [0,\vartheta)$ if
  $\vartheta < 1$ and $\theta \in [0,1]$ if $\vartheta=1$, let
  $\eta \in C(\overline I,C^{0,\vartheta}(\overline\Omega))$ be
  positive.  Then the operator $-\nabla \cdot \eta \mu \nabla+1$
  exhibits nonautonomous maximal parabolic regularity on $X_\theta$
  with domain $Y_\theta$, that is, for any $r \in (1,\infty)$, if
  $f \in L^r(I,X_\theta)$ and $w_0 \in (X_\theta,Y_\theta)_{1/r',r}$,
  then there is a unique solution $w$
  to~\eqref{eq::nonautomous-general-problem-scale} and the
  corresponding solution map
  \begin{align}\label{eq::iso}
    L^r(I,X_\theta) \times (X_\theta,Y_\theta)_{1/r',r} \to \W^{1,r}(I,(X_\theta,Y_\theta)), \qquad (f,w_0) \mapsto w,
  \end{align}
  is a topological isomorphism. If
  $\Ekal \subset C(\overline I,C^{0,\vartheta}(\overline\Omega))$ is a
  compact set such that $\eta(x) > 0$ for all $x \in \overline{Q}$ for
  all $\eta \in \Ekal$, then the isomorphism~\eqref{eq::iso} is
  equicontinuous with respect to $\eta \in \Ekal$.
\end{theorem}
 
\begin{proof}
Due to Lemma~\ref{lem::constdomain} the map
\[
  \overline I \to \Lkal(Y_\theta, X_\theta), \qquad t \mapsto -\nabla
  \cdot \eta(t) \mu \nabla + 1,
\]
is well-defined and continuous, and for every $t$, the resulting
operator is even continuously invertible. This together with
Corollary~\ref{cor::max-reg-interpolation-space} allows to invoke
Proposition~\ref{prop::pruess-schnaubelt-cont-MPR} which yields the
claim on nonautonomous maximal parabolic regularity on $X_\theta$ with
the domain $Y_\theta$.

  It remains to argue on equicontinuity of the solution maps for
  $\eta \in \Ekal$: Note that for each $r \in (1,\infty)$ the map
  \begin{align*}
    C(\overline I,C^{0,\vartheta}(\overline\Omega)) \supset \Ekal &\to
    \mathcal{L}\bigl(\W^{1,r}(I,(X_\theta,Y_\theta)), L^r(I,X_\theta)
    \times (X_\theta,Y_\theta)_{1/r',r}\bigr), \\[0.25em] 
    \eta &\mapsto (\partial - \nabla \cdot \eta \mu \nabla + 1, \gamma_0)^{-1}, 
  \end{align*}
  is well-defined due to nonautonomous parabolic regularity of
  $-\nabla \cdot \eta \mu \nabla + 1$ in $X_\theta$ with domain
  $Y_\theta$ for each $\eta \in \Ekal$. Moreover, the map is
  continuous by Lemma~\ref{lem::constdomain} and continuity of
  operator inversion $A \mapsto A^{-1}$. Hence, its operator norm is bounded on the compact set $\Ekal$.
\end{proof}


\begin{remark}\label{rem::mpr}
  The significant difference between the proof of
  Theorem~\ref{thm::mpronxtheta} and the one of~\cite[Theorem
  3.20]{Bonifacius2018} is that here, by virtue of
  Lemma~\ref{lem::constdomain}, the domains of
  $-\nabla\cdot\eta(t)\mu\nabla + 1$ in $X_\theta$ are \emph{a priori}
  independent of $t$. (It follows \emph{a posteriori} in the proof
  of~\cite[Theorem~3.20]{Bonifacius2018}, though.) This allows to
  apply Proposition~\ref{prop::pruess-schnaubelt-cont-MPR} basing on
  constant domains which is much less involved than the treatment for
  \emph{a priori} nonconstant domains based on uniform
  $\Rkal$-boundedness as in~\cite{Portal2006}. The present insight
  that the domains \emph{are} constant \emph{a priori} is in turn based on
  the bilinear interpolation technique as established in
  Appendix~\ref{sec::RehbergMeinlschmidt} and invariance of
  H\"older-continuous functions under taking the reciprocal.
\end{remark}

\section{Abstract quasilinear parabolic equations in the scale}
\label{subsec::eqonXtheta}

Building upon on nonautonomous maximal parabolic regularity
for~\eqref{eq::nonautomous-general-problem-scale} on $X_\theta$ as in
Theorem~\ref{thm::mpronxtheta}, we now treat the quasilinear abstract
problem~\eqref{eq::main-problem-abstract} on $X_\theta$ utilizing a
bootstrapping argument. This is the main achievement of the paper. The
bootstrapping procedure starts from Theorem~\ref{thm::exonW1p} for
$X_0 = W^{-1,p}_D$ and works its way up through $X_\theta$ until some
maximal $\bar\theta \in (0,1]$ is reached. The fundamental idea is
that given the solution $u$ in $X_\theta$, one finds---via
Theorem~\ref{thm::mpronxtheta}---that the associated degree of
H\"older continuity of $u$ is in fact sufficient to obtain
nonautonomous maximal parabolic regularity for
$-\nabla\cdot\xi(u)\mu\nabla + 1$ in $X_{\theta^+}$ for
$\theta^+ > \theta$.

The value of this maximal $\bar\theta$ will be related to the degree
of Hölder-regularity $\kappa$ of the elliptic boundary value problem
in $L^p$ associated to $-\nabla \cdot \mu \nabla+1$. In view of
Theorem~\ref{thm::mpronxtheta}, this should not come as a surprise,
since the highest degree of H\"older-regularity obtainable within the
$X_\theta$ scale will be limited exactly by the associated value for
$X_1 = L^p$.

Thus, we start with a short discussion regarding $\kappa$. The main
argument is then given for the case where $\Fkal$ in the abstract
quasilinear equation~\eqref{eq::main-problem-abstract} is
\emph{constant} in order to make the idea more transparent. This is
Section~\ref{subsec::simplecase}. We then include $\Fkal$ in
Section~\ref{subsec::generalcase}.


\begin{definition}\label{def::regdom} 
  Let $\kappa \in [1-\frac{d}{p},2-\frac{d}{p}]$ be as large as
  possible such that
  \begin{equation*}
    \Dom_{L^p}(-\nabla \cdot \mu \nabla+1) \embeds
    \begin{cases}
      C^{\lfloor\kappa\rfloor,\kappa}(\overline \Omega) &
      \text{if}~\kappa\neq 1,\\[0.25em] C^{0,1}(\overline\Omega) &
      \text{if}~\kappa =1.
    \end{cases}
  \end{equation*}
\end{definition}
Let us point out that such a $\kappa \geq 1-\frac{d}p$ always exists under
Assumption~\ref{ass::problem-data}~(Iso), since
\begin{equation*}
  \Dom_{L^p}(-\nabla \cdot \mu \nabla+1) \embeds W^{1,p}_D \embeds
  C^{0,1-\frac{d}p}(\overline\Omega). 
\end{equation*}

\begin{remark}
  \label{rem:kappa-assumption}
  We comment on particular cases for $\kappa$ in
  Definition~\ref{def::regdom}.
  \begin{enumerate}[(1)]
  \item The optimal case $\kappa =2-\frac{d}{p}$ is obtained for
    example when $\Omega$ is a $C^{1,1}$-domain, $\mu$ is
    Lipschitz-continuous, and pure homogeneous Dirichlet or Neumann
    boundary conditions are imposed, so either
    $\Gamma_D = \partial \Omega$ or $\Gamma_D = \emptyset$. In this
    case we have optimal $W^{2,p}$-regularity for the boundary
    problem associated to $-\nabla\cdot\mu\nabla + 1$, so
    $\Dom_{L^p}(-\nabla \cdot \mu \nabla+1) = W^{2,p} \cap W^{1,p}_D$,
    see~\cite[Theorems~2.4.2.5 and~2.4.2.7]{Grisvard1985}.
  \item Suppose that the domain geometry remains rather irregular, but
    $\mu$ gives rise to a multiplier on $H^{\eta,p}$ for some
    $\eta < 1/p$, possibly small.  Then it can be shown that a
    fractional power of the $L^p$-realization of
    $-\nabla \cdot\mu\nabla + 1$ will have the domain $H^{1+\eps,p}_D$
    for some $\eps < \eta$ by extrapolation techniques,
    see~\cite[Section~3]{Meinlschmidt2021}. For example, $\mu$ could
    be the characteristic function of a convex subset of $\Omega$, or
    a subset of locally bounded perimeter, or a linear combination of
    such functions. We do not go into details here. In any case, it
    follows that we can have $\kappa = 1+\eps-\frac{d}p>1-\frac{d}p$
    in this quite general situation.
  \end{enumerate}
\end{remark}

The maximal possible value of $\kappa$ is nontrivial to determine; it
is related to the regularity of the boundary of $\Omega$, the boundary
conditions imposed, the coefficient function $\mu$, and $p$, and may
not be known explicitely in general. We thus formulate our further
results in dependence of $\kappa$.  See also
Appendix~\ref{sec:hold-regul-ellipt} for a further discussion for
various practical situations.

\subsection{The interpolation spaces and fractional powers}
\label{subsec::interpolation-spaces}
In this section we collect some results on the interpolation spaces
\begin{equation*}
  \bigl[X_\vartheta,Y_\vartheta\bigr]_{\tau} \coloneqq  \left[\bigl[W^{-1,p}_D,
    L^p\bigr]_{\vartheta}, \Dom_{X_\vartheta}(-\nabla \cdot \mu \nabla +
    1)\right]_{\tau}.
\end{equation*}
These occur naturally in the consideration of quasilinear problems
from trace embeddings, see Lemma~\ref{lem::max-reg-spaces-embed};
usually, they are considered in order to establish regularity
properties for the nonlinear functions in fixed point arguments. In
view of the results of Section~\ref{sec::quasilinonXtheta} on
nonautonomous maximal parabolic regularity for H\"older-continuous
coefficients, it is thus not surprising that we aim for an embedding
into a suitable H\"older space for $[X_\vartheta,Y_\vartheta]_{\tau}$.

The first auxiliary result is that the spaces in question can be
identified with the domain of a fractional power of
$-\nabla\cdot\mu\nabla + 1$, co-restricted to $L^p$. 

\begin{lemma}
  \label{lem:interpolation-is-fracpower-domain}
  Let $\vartheta \in [0,1]$ and
  $\tau \in (\frac12-\frac\vartheta2,1]$. Then
  \begin{equation*}
    \bigl[X_\vartheta,Y_\vartheta\bigr]_{\tau} = \Dom_{L^p}(-\nabla
    \cdot\mu\nabla + 1)^{\tau-\frac12+\frac\vartheta2}.
  \end{equation*}
\end{lemma}

\begin{proof}
  Recall that the fractional powers of $-\nabla\cdot\mu\nabla + 1$
  were well-defined and that we even have bounded imaginary powers at
  our disposal
  (Proposition~\ref{prop::divergence-operator-properties-no-iso-yet}).
  We thus apply Lemma~\ref{lem::domain2} and
  $\Dom_{W^{-1,p}_D}(-\nabla \cdot \mu \nabla + 1) = W^{1,p}_D$ by
  Assumption~\ref{ass::problem-data}~(Iso):
  \begin{align}\label{eq::lemhelp1}
    \bigl[X_\vartheta,Y_\vartheta\bigr]_{\tau} &=
    \left[ \bigl[W^{-1,p}_D, W^{1,p}_D\bigr]_\tau, \bigl[L^p, \Dom_{L^p}(-\nabla \cdot \mu \nabla + 1)\bigr]_\tau\right] _\vartheta.
  \end{align} Next, we branch along
  $\tau = \frac12$. If $\frac12-\frac\vartheta2 < \tau < \frac12$,
  then necessarily $\vartheta>0$ and
  we argue as follows:
  \begin{multline*}
    \left[ \bigl[W^{-1,p}_D, W^{1,p}_D\bigr]_\tau, \bigl[L^p,
      \Dom_{L^p}(-\nabla \cdot \mu \nabla + 1)\bigr]_\tau\right]
    _\vartheta \\ = \left[ \bigl[W^{-1,p}_D, W^{1,p}_D\bigr]_\tau,
      \bigl[L^p, W^{1,p}_D\bigr]_{2\tau}\right] _\vartheta = \left[
      \bigl[W^{-1,p}_D, L^p\bigr]_\xi, W^{1,p}_D\right]
    _{(1+\vartheta)\tau}
  \end{multline*}
  where we have used the Kato square root property
  $\Dom_{L^p}(-\nabla\cdot\mu\nabla + 1)^{\frac12} = W^{1,p}_D$
  (Proposition~\ref{prop::kato-square-root}) and the reiteration
  theorem for complex interpolation with
  $(1-(1+\vartheta)\tau)\xi = (1-2\tau)\vartheta$.  Now,
  $(1+\vartheta)\tau - (1-(1-\vartheta)\tau)(1-\xi) =
  2\tau-1+\vartheta > 0$. Hence, using
  Lemma~\ref{lem::interpol-Bessel} and the reiteration theorem twice,
  \begin{align*}
    \left[
      \bigl[W^{-1,p}_D, L^p\bigr]_\xi, W^{1,p}_D\right]
    _{(1+\vartheta)\tau} &=  \left[
      \Bigl[W^{-1,p}_D, \bigl[W^{-1,p}_D,W^{1,p}_D\bigr]_{\frac12}
      \Bigr]_\xi, W^{1,p}_D\right] 
    _{(1+\vartheta)\tau} \\ & =
    \bigl[W^{-1,p}_D,W^{1,p}_D\bigr]_{\tau+\frac\vartheta2} \\ &
    = \bigl[
    L^p, W^{1,p}_D\bigr]_{2\tau-1+\vartheta} =
    \Dom_{L^p}(-\nabla\cdot\mu\nabla + 1)^{\tau-\frac12+\frac\vartheta2}.
  \end{align*}
  Next, let $\tau \geq \frac12$ and $\tau > \frac12$ if
  $\vartheta=0$. Then, again by reiteration and
  Lemma~\ref{lem::interpol-Bessel} we have
  \begin{equation*}
    \bigl[W^{-1,p}_D,W^{1,p}_D\bigr]_{\tau} =
    \left[\bigl[W^{-1,p}_D,W^{1,p}_D\bigr]_{\frac12},W^{1,p}_D\right]_{2\tau-1}
    =  \bigl[L^p,W^{1,p}_D\bigr]_{2\tau-1}.
  \end{equation*}
  Again, the Kato square root property implies that
  $\Dom_{L^p}(-\nabla\cdot\mu\nabla + 1)^{\frac12} =
  W^{1,p}_D$. Hence, by reiteration for fractional power domains
  (Lemma~\ref{lem::reiteration-frac-powers}) in~\eqref{eq::lemhelp1},
  \begin{align*}
    \bigl[X_\vartheta,Y_\vartheta\bigr]_{\tau} = & \left[
      \bigl[L^p,W^{1,p}_D\bigr]_{2\tau-1}, \bigl[L^p, \Dom_{L^p}(-\nabla
      \cdot \mu \nabla + 1)\bigr]_\tau\right] _\vartheta \\  =&
    \left[ \Dom_{L^p}(-\nabla\cdot\mu\nabla + 1)^{\tau-\frac12},
      \Dom_{L^p}(-\nabla \cdot \mu \nabla + 1)^\tau\right] _\vartheta \\ 
    =& \,\Dom_{L^p}(-\nabla\cdot\mu\nabla + 1)^{\tau-\frac12 +\frac\vartheta2}.
  \end{align*}
  This was the claim.
\end{proof}

\begin{lemma}
  \label{lem:frac-power-embed-hoelder}
  Let $\vartheta \in [0,1]$ and
  $\tau \in (\frac12+\frac{d}{2p}-\frac{\vartheta}2, 1]$. Assume that
  $\kappa \leq 1$ in Definition~\ref{def::regdom}. Set
  \begin{equation*}
    \sigma \coloneqq 2\tau-1+\vartheta-\frac{d}p \qquad \text{if}
    \qquad \tau  < 1-\frac\vartheta2
  \end{equation*}
  and
  \begin{equation*}
    \sigma \coloneqq 1-\frac{d}p +
    \bigl(2\tau-2+\vartheta\bigr)\Bigl(\kappa - 1 + \frac{d}p\Bigr)
    \qquad \text{if} \qquad \tau  \geq 1-\frac\vartheta2.
  \end{equation*}
  Then
  \begin{equation*}
    \bigl[X_\vartheta,Y_\vartheta\bigr]_{\tau} \embeds C^{0,\sigma}(\overline\Omega).
  \end{equation*}
\end{lemma}

\begin{proof}
  We argue via Lemma~\ref{lem:interpolation-is-fracpower-domain}. Let
  first $\tau < 1-\frac\vartheta2$. Then
  $\tau-\frac12+\frac\vartheta2 < \frac12$ and we find with Sobolev
  embedding (Lemma~\ref{lem::Sobolev-embeddings})
  \begin{equation*}
    \Dom_{L^p}(-\nabla\cdot\mu\nabla +
    1)^{\tau-\frac12+\frac\vartheta2} = H^{2\tau-1+\vartheta,p}_D
    \embeds C^{0,2\tau-1+\vartheta - \frac{d}p} = C^{0,\sigma}(\overline\Omega).
  \end{equation*}
  Now turn to $\tau \geq 1-\frac\vartheta2$. Then
  $\tau-\frac12+\frac\vartheta2 \geq \frac12$, hence,
  \begin{equation*}
    \Dom_{L^p}(-\nabla \cdot \mu \nabla
    +1)^{\tau-\frac12+\frac\vartheta2} = \left[\Dom_{L^p}(-\nabla \cdot \mu \nabla
      +1)^{\frac12},\Dom_{L^p}(-\nabla \cdot \mu \nabla
      +1)\right]_{2\tau-2+\vartheta}
  \end{equation*}
  and so, via Definition~\ref{def::regdom} for $\kappa$, the
  assumption $\kappa \leq 1$ and Proposition~\ref{prop::interpolationhoelder},
  \begin{align}\label{eq::fracpower-rewrite}
    \Dom_{L^p}(-\nabla \cdot \mu \nabla
    +1)^{\tau-\frac12+\frac\vartheta2}  & = \left[W^{1,p}_D,\Dom_{L^p}(-\nabla \cdot \mu \nabla
      +1)\right]_{2\tau-2+\vartheta} \\ & \embeds
    \bigl[C^{0,1-\frac{d}p}(\overline\Omega),
    C^{0,\kappa}(\overline\Omega)\bigr]_{2\tau-2+\vartheta} 
    \notag \\ & \embeds 
    \bigl(C^{0,1-\frac{d}p}(\overline\Omega),
    C^{0,\kappa}(\overline\Omega)\bigr)_{2\tau-2+\vartheta,\infty}   
    = C^{0,\sigma}(\overline\Omega).\notag\qedhere 
  \end{align}
\end{proof}

Unfortunately we did not succeed in proving an analogue to
Lemma~\ref{lem:frac-power-embed-hoelder} for the case $\kappa > 1$
under our general assumptions. As one sees
from~\eqref{eq::fracpower-rewrite}, with our technique the question
boils down to an interpolation result of the form
\begin{equation*}
  \bigl(C^{0,1-\frac{d}p}(\overline\Omega),
  C^{1,\kappa-1}(\overline\Omega)\bigr)_{\xi,\infty} \embeds 
  C^{0,1}(\overline\Omega)  \qquad \text{when} \qquad (1-\xi)(1-\tfrac{d}p) + \xi\kappa > 1.
\end{equation*}
We conjecture that such a result will be true for our case of a weak
Lipschitz domain. We can indeed prove it for example when $\Omega$
admits a $C^1$-boundary by extension techniques,
see~\cite[Example~1.1.9]{Lunardi2009}. This is also in line with the
rule of thumb that $\kappa > 1$ in Definition~\ref{def::regdom} will
require a more smooth situation than we generally suppose within this
work. For now we content ourselves with
making the following
\begin{assumption}
  \label{ass::fracpower}
  If $\kappa > 1$ in Definition~\ref{def::regdom}, we assume that we have for
  $\vartheta \in [0,1]$ and $\tau \in [1-\frac\vartheta2,1]$,
  \begin{equation*}
    \Dom_{L^p}(-\nabla \cdot \mu \nabla
    +1)^{\tau-\frac12+\frac\vartheta2} \embeds C^{0,\sigma}(\overline\Omega)
  \end{equation*}
  with
  \begin{equation*}
    \sigma \coloneqq 1 \wedge \bigl(1-\tfrac{d}p + (2\tau-2+\vartheta)(\kappa-1+\tfrac{d}p)\bigr).
  \end{equation*}
\end{assumption}
\begin{remark}
  \label{rem::fracpowerass}
  If in fact the optimal regularity
  \begin{equation*}
    \Dom_{L^p}(-\nabla \cdot \mu \nabla
    +1) = W^{2,p} \cap W^{1,p}_D
  \end{equation*}
  is valid, then we can circumvent the H\"older interpolation problem
  and Assumption~\ref{ass::fracpower} is indeed true. In this case,
  $\kappa = 2-\frac{d}{p}$ by Sobolev embedding, so
  in~\eqref{eq::fracpower-rewrite} we obtain
  \begin{align*}
    \Dom_{L^p}(-\nabla \cdot \mu \nabla
    +1)^{\tau-\frac12+\frac\vartheta2}  & =
    \left[W^{1,p}_D,\Dom_{L^p}(-\nabla \cdot \mu \nabla
      +1)\right]_{2\tau-2+\vartheta}
    \\ & = H^{2\tau+\vartheta-1,p} \cap
    W^{1,p}_D \embeds C^{0,\sigma}(\overline\Omega).
  \end{align*}
  (See the extension arguments in
  Remark~\ref{ex::convex-domain-regularity} together
  with~\cite[Lemma~3.1 and~Corollary~3.7]{Bechtel2019} to justify
  interpolation and Sobolev embeddings here.) Thus,
  Assumption~\ref{ass::fracpower} in fact only covers the quite
  particular case where we do not have optimal Sobolev regularity for
  $-\nabla\cdot\mu\nabla + 1$ in $L^p$, but still $\kappa > 1$ in
  Definition~\ref{def::regdom}. 
\end{remark}

\subsection{Regularity bootstrapping: simple case}
\label{subsec::simplecase}
We now present the bootstrapping procedure to obtain \emph{global}
solutions of~\eqref{eq::main-problem-abstract} on the
$X_\theta$-scale. To make the ideas more transparent, we first assume
that $\alpha \equiv 0$ and that $\Fkal \equiv f$ is constant,
resulting in the following problem:
\begin{align}\label{eq::main-problem-abstract-simple}
  \left\{ \qquad \begin{aligned}
      \partial_t u - \nabla \cdot \xi(u) \mu \nabla u + u &= f
      &&\text{in } L^s(I,X_\theta), \\ 
      u(0) &= u_0 && \text{on } (X_\theta,Y_\theta)_{1/s',s}. \\ 
    \end{aligned} \right.
\end{align}
Note that this problem is still quasilinear and therefore shares major
characteristic difficulties with~\eqref{eq::main-problem-abstract}. We
will come back to the full equation~\eqref{eq::main-problem-abstract}
in the subsequent section. The deciding factor on the size of $\theta$
that we can afford is the number
\begin{equation*}
  \bar\theta \coloneqq \frac{1-\frac{d}p}{2-\frac{d}p-\kappa}.
\end{equation*}
Note that $\bar \theta \geq 1-\frac{d}{p}$; moreover,
$\bar\theta \geq 1$, that is, we can reach $X_1 = L^p$, if and only if
$\kappa \geq 1$. With this, the main regularity result
for~\eqref{eq::main-problem-abstract-simple}, and thereby the first
version of the main theorem, reads as follows:

\begin{theorem}\label{thm::eqinXtheta}
  Adopt Assumption~\ref{ass::fracpower} if $\kappa>1$ and fix
  $\theta\in [0,\bar\theta) \cap [0,1]$. Suppose that
  \begin{equation*}
    \frac1s < \frac12-\frac{d}{2p} \quad \text{if} \quad \theta < 1-\frac{d}p \qquad \text{and} \qquad \frac{1}{s} <
    \frac{1-\frac{d}{p} -
      \theta}{2(\kappa-1+\frac{d}p)} + \frac\theta2 \quad \text{otherwise}.
  \end{equation*}
  If $f \in L^s(I,X_\theta)$ and
  $u_0 \in (X_\theta, Y_\theta)_{1/s',s}$, then the unique solution
  $u \in \W^{1,s}(I,(X_0,Y_0))$ of~\eqref{eq::main-problem-abstract-simple} on $X_0$ in fact
  satisfies $u \in \W^{1,s}(I,(X_\theta,Y_\theta))$, in particular,
  $u \in C^{\beta}(\overline I,C^{0,\theta}(\overline\Omega))$, and
  it is the unique solution to~\eqref{eq::main-problem-abstract-simple} on $X_\theta$. Moreover,
  the solution map $(u_0,f) \mapsto u$ associated with~\eqref{eq::main-problem-abstract-simple}
  maps bounded sets in
  $(X_\theta,Y_\theta)_{1/s',s} \times L^s(I,X_\theta)$ into bounded
  sets in $\W^{1,s}(I,(X_\theta,Y_\theta))$.
\end{theorem}


To make the structure of the proof of Theorem~\ref{thm::eqinXtheta}
more transparent, we first state the required H\"older embedding as an
auxiliary result. Of course we use $\kappa$ from
Definition~\ref{def::regdom}.  Note that the second case in the
following lemma is void if $\kappa = 1-\frac{d}p$.

\begin{lemma}\label{lem::embeddings}
  Suppose that Assumption~\ref{ass::fracpower} is valid if $\kappa>1$. For
  $\vartheta \in [0,1]$, let either
  \begin{equation*}
    0 \leq \sigma <  1-\frac{d}{p}
    \qquad \text{and} \qquad \frac1s < \frac{1-\frac{d}p - \sigma}2 + \frac\vartheta2.
  \end{equation*}
  or
  \begin{equation*}
    \sigma \in \Bigl[1-\tfrac{d}p,1-\tfrac{d}{p} + \vartheta\bigl(\kappa -
    1 + \tfrac{d}{p}\bigr)\Bigr) \cap [0,1] \qquad \text{and} \qquad \frac1s <
    \frac{1-\frac{d}{p} -
      \sigma}{2(\kappa-1+\frac{d}p)} + \frac\vartheta2.
  \end{equation*}
  Then
  $\W^{1,s}(I,(X_\vartheta,Y_\vartheta)) \embeds_c C(\overline I,
  C^{0,\sigma}(\overline\Omega))$.
\end{lemma}

\begin{proof}
  By Bochner-Sobolev embeddings, see Lemma~\ref{lem::max-reg-spaces-embed},
  we get
  \begin{align}\label{eq::maxreg-embed-theta-cont}
    \W^{1,s}\bigl(I,(X_\vartheta,Y_\vartheta)\bigr) \embeds_c
    C\bigl(\overline I,
    [X_\vartheta,Y_\vartheta]_\tau\bigr),
  \end{align}
  as long as $0 \leq \tau < 1-\frac{1}{s}$. (Note that
  $Y_\theta \embeds W^{1,p}_D \embeds_c L^p \embeds
  X_\theta$.)
  Thus, the claim follows via Lemma~\ref{lem:frac-power-embed-hoelder}
  in the following cases:
  \begin{equation*}
    \sigma < 1 - \frac{d}p \qquad \text{and} \qquad \tau \coloneqq \frac12 + \frac{d}{2p} - \frac\vartheta2 +
    \frac\sigma2,
  \end{equation*}
  or (here we use Assumption~\ref{ass::fracpower} when $\kappa > 1$)
  \begin{equation*}
    \sigma \in \Bigl[1-\tfrac{d}p,1-\tfrac{d}{p} + \vartheta\bigl(\kappa -
    1 + \tfrac{d}{p}\bigr)\Bigr) \cap [0,1] \qquad \text{and} \qquad  \tau \coloneqq
    1-\frac\vartheta2 +
    \frac{\sigma-1+\frac{d}p}{2(\kappa-1+\frac{d}p)}. 
  \end{equation*}
  The conditions on $s$ in the statement correspond exactly to
  $\frac1s<1-\tau$, so $\tau < 1-\frac1s$, for each case, hence the
  choices of $\tau$ and $s$ are compatible and we obtain the
  assertion.
\end{proof}

We now come back to proving Theorem~\ref{thm::eqinXtheta}.

\begin{proof}[Proof of Theorem~\ref{thm::eqinXtheta}]
  We start with $\vartheta = 0$. With Theorem~\ref{thm::exonW1p} there
  exists a unique solution $u \in \W^{1,s}(I,(X_0,Y_0))$
  to~\eqref{eq::main-problem-abstract-simple} on $X_0$. Now fix $\epsilon > 0$ such that
  \begin{equation}\label{eq::defeps}
    \frac1s + \frac\eps2 < \frac12-\frac{d}{2p} \qquad \text{if} \quad
    \theta < 1-\frac{d}p
  \end{equation}
  and
  \begin{equation}
    \label{eq::defeps-alt}
    \frac{1}{s} + \frac\eps{2(\kappa-1+\frac{d}p)} <
    \frac{1-\frac{d}{p} -
      \theta}{2(\kappa-1+\frac{d}p)} + \frac\theta2 \qquad \text{if}
    \quad \theta \geq 1-\frac{d}p.
  \end{equation}
  We proceed iteratively as follows:
  \begin{enumerate}[(1)]
  \item Suppose that $u \in \W^{1,s}(I,(X_\vartheta,Y_\vartheta))$ for
    some $\vartheta \in [0,\theta)$. By Lemma~\ref{lem::embeddings},
    we have $u \in C(\overline I, C^{0,\sigma}(\overline\Omega)))$ for
    $\sigma \coloneqq \vartheta + \epsilon$ by the choices of $\eps$
    in~\eqref{eq::defeps}/\eqref{eq::defeps-alt} and of $s$. Note here
    that for $\theta \geq 1-\frac{d}p$, with the choice of $\eps$ as in~\eqref{eq::defeps-alt}, we have
    \begin{equation*}
      \frac1s + \frac\eps2 \leq
      \frac{1}{s} + \frac\eps{2(\kappa-1+\frac{d}p)} <
      \frac{1-\frac{d}{p} -
        \theta}{2(\kappa-1+\frac{d}p)} + \frac\theta2 \leq \frac12-\frac{d}{2p},
    \end{equation*}
    so $\eps$ is also valid for~\eqref{eq::defeps}.

  \item Set
    $\vartheta^+ \coloneqq \min(\vartheta +
    \frac{\epsilon}{2},\theta)$. Then Theorem~\ref{thm::mpronxtheta}
    with $\eta = \xi(u)$ yields nonautonomous maximal parabolic
    regularity of $-\nabla \cdot \xi(u)\mu \nabla + 1$ on
    $X_{\vartheta^+}$. Since
    $u_0 \in (X_\theta,Y_\theta)_{1/s',s} \embeds
    (X_{\vartheta^+},Y_{\vartheta^+})_{1/s',s}$ and
    $f \in L^s(I,X_\theta) \embeds L^s(I,X_{\vartheta^+})$ by
    assumption, we conclude that indeed
    $u \in \W^{1,s}(I,(X_{\vartheta^+},Y_{\vartheta^+}))$ is the unique
    solution to~\eqref{eq::main-problem-abstract-simple} in $X_{\vartheta^+}$.
 
  \item Note that $\eps$ is independent of the current $\vartheta$,
    so we can repeat the foregoing procedure a finite number of times until
    $\vartheta^+ = \theta$ is reached.
  \end{enumerate}
  The H\"older regularity stated for $u$ follows immediately from the
  construction, see also Lemma~\ref{lem::embeddings} and the choices
  of $s$ and $\theta$.
  
  Finally, we argue on the statement on mapping bounded sets into
  bounded sets. Such a result is already shown
  in~\cite[Corollary~5.8]{Meinlschmidt2016} for $\vartheta =
  0$. Furthermore, the embeddings of
  $\W^{1,s}(I,(X_\vartheta,Y_\vartheta))$ into $C(\overline I,C^{0,\sigma}(\overline\Omega))$
  which were utilized along the way are \emph{all} compact, cf.\
  Lemma~\ref{lem::embeddings}. Thus, the uniformity statement in
  Theorem~\ref{thm::mpronxtheta} is carried through every step of the
  bootstrapping argument.
\end{proof}

The above theorem may be regarded as
an extension of~\cite[Theorem~3.20]{Bonifacius2018}:

\begin{remark}
  Let $\kappa = 1-\frac{d}{p}$. Then we have
  $\theta < \bar\theta = 1-\frac{d}{p}$. Theorem~\ref{thm::eqinXtheta}
  yields optimal regularity solutions of~\eqref{eq::main-problem-abstract} on
  \[
    X_{\theta} = \bigl[W^{-1,p}_D,L^p\bigr]_{\theta} = H^{-\zeta,p}_D
  \]
  with $-\zeta = \theta-1$, i.e.,
  $\zeta = 1-\theta > 1-(1-\frac{d}{p}) = \frac{d}{p}$. This yields the same
  assumption on $\zeta$ as in~\cite{Bonifacius2018}. (Due to $p \leq d+1$, we
  also have $\theta = 1-\zeta < 1/p$ in this case.) For the same value of
  $\zeta$, we do however require less integrability for $s$ since with our
  technique we do not have to pass through $u(t) \in W^{1,p}_D$ as
  in~\cite[Proof of Thm.~3.20, Step~3]{Bonifacius2018}. Therefore, without
  imposing \emph{any} new assumption on $\kappa$, Theorem~\ref{thm::eqinXtheta}
  reproduces and improves the result from~\cite[Theorem 3.20]{Bonifacius2018}
  which was obtained there by much more involved reasoning.
\end{remark}

\subsection{Regularity bootstrapping: general case}
\label{subsec::generalcase}

Essentially the same bootstrapping technique as the one used in the
proof of Theorem~\ref{thm::eqinXtheta} also yields unique solutions of
the full problem~\eqref{eq::main-problem-abstract} in the $X_\theta$-scale. This is the
second version of our main result.

\begin{theorem}\label{thm::regforeq}
  Let $\theta$ and $s$ be as in Theorem~\ref{thm::eqinXtheta} and let
  $u_0 \in (X_\theta,Y_\theta)_{1/s',s}$. If $\theta \geq \frac1p$,
  suppose that $\alpha \equiv 0$. Assume that there is
  $\epsilon_0 > 0$ such that for all $0 < \epsilon < \epsilon_0$ and
  all $\vartheta \in [0,\theta]$ the implication
  \begin{align}\label{eq::cond-on-Fkal}
    \Bigl[w \in \W^{1,s}\bigl(I,(X_\vartheta,Y_\vartheta)\bigr) \quad \implies \quad \Fkal(w) \in L^s\bigl(I,X_{\min(\theta,\vartheta+\frac{\epsilon}{2})}\bigr)\Bigr]
  \end{align}
  holds true. Then, the unique solution
  $u \in \W^{1,s}(I,(X_0,Y_0))$ of~\eqref{eq::main-problem-abstract} on $X_0$ in fact
  satisfies $u \in \W^{1,s}(I,(X_\theta,Y_\theta))$, in particular,
  $u \in C^{\beta}(\overline I,C^{0,\theta}(\overline\Omega))$, and
  it is the unique solution to~\eqref{eq::main-problem-abstract} on $X_\theta$. 
\end{theorem}

\begin{proof}
  Since the main line of the proof is the same as for
  Theorem~\ref{thm::eqinXtheta}, we only comment on the changes: The
  fixed right hand side $f$ in~\eqref{eq::main-problem-abstract-simple} is replaced by
  \begin{equation*}
    g \coloneqq \Fkal(u)-\Bkal_\alpha u.
  \end{equation*}
  Consequently, we have to make sure that if
  $u\in\W^{1,s}(I,(X_\vartheta,Y_\vartheta))$, then
  $g \in L^s(I,X_{\vartheta^+})$ in order to go through step (ii) of
  the bootstrapping argument as before. For $\Fkal(u)$, this is
  ensured exactly by the assumed property~\eqref{eq::cond-on-Fkal},
  while we refer to Lemma~\ref{lem::trace} for $\Bkal_\alpha u$;
  recall~\eqref{eq:maxreg-is-hoelder}.  The latter term
  $\Bkal_\alpha u$ is void if $\theta \geq \frac1p$.
\end{proof}

\begin{corollary}
  \label{cor::regforeq-uniform}
  Adopt the assumptions of Theorem~\ref{thm::regforeq}, but let
  $\alpha \equiv 0$ also when $\theta < \frac1p$. Denote by $\Gkal$
  the family of nonlinear functions $\Fkal$ satisfying the standing
  Assumption~\ref{ass::problem-data} and
  assumption~\eqref{eq::cond-on-Fkal}, and for which in addition there is a
  constant $C_{\Gkal} > 0$ for which
  \[
     \sup_{w \in
      \W^{1,s}(I,(X_\theta,Y_\theta))}\bigl\lVert \Fkal(w)
    \bigr\rVert_{L^s(I,X_\theta)} \leq
    C_\Gkal.
  \]
  Denote by $u_\Fkal$ the unique solution
  $u \in \W^{1,s}(I,(X_\theta,Y_\theta))$ to~\eqref{eq::main-problem-abstract} supplied by
  Theorem~\ref{thm::regforeq}. Then, the set
  $\{u_\Fkal\colon\Fkal\in\Gkal\}$ is bounded in
  $\W^{1,s}(I,(X_\theta,Y_\theta))$.
\end{corollary}

\begin{proof}
  This follows directly from the statement about the solution operator
  mapping bounded sets into bounded sets in
  Theorem~\ref{thm::eqinXtheta}, with $f = \Fkal(u)$.
\end{proof}

\begin{remark}
  We mention a few examples for nonlinearities fulfilling the
  assumption in Theorem~\ref{thm::regforeq}:
  \begin{enumerate}[(1)]
  \item Any nonlinear function given by a Nemytskii operator induced
    by a sufficiently regular real function with suitable growth
    bounds---or a bounded one--- and also suitably monotone ones will do the trick. Analogous boundary
    terms can be treated only for $\theta < \frac{1}{p}$.
 
  \item Similarly, the nonlocal term from
    Remark~\ref{rem::assumptdisc_2} obviously
    fulfills~\eqref{eq::cond-on-Fkal}.
 
  \item A drift-type term as in Remark~\ref{rem::assumptdisc_2} is
    also covered by the assumptions, if the following additional
    (sufficient) conditions are fulfilled:
    \begin{align*}
      g \in L^s(I,H^{1+\theta,p}), \quad \text{and} \quad 
      \mu \in C^{0,\min(1,\tau)}(\overline\Omega)^{d\times d} \text{ for some } \tau > \theta.
    \end{align*}
    Let us verify~\eqref{eq::cond-on-Fkal}. We are quite concise and borrow
    several arguments from~\cite[Section~3]{Meinlschmidt2021}. Suppose first that
    $\vartheta+\frac\eps2 \leq \theta$. For almost every $t \in I$, we
    have
    $\nabla g(t) \in H^{\theta,p} \embeds
    H^{\vartheta+\frac\eps2,p}$. On the other hand, from
    $w \in \W^{1,s}(I,(X_\vartheta,Y_\vartheta))$ it follows that
    $w \in C(\overline I, C^{0,\vartheta+\epsilon})$ by
    Lemma~\ref{lem:frac-power-embed-hoelder} as in the proof of
    Theorem~\ref{thm::regforeq}. But then $w(t)\mu$ is H\"older
    continuous of degree $> \vartheta+\frac\eps2$ and thus a
    multiplier on $H^{\vartheta+\frac\eps2,p}$ for every $t \in
    \overline I$. Therefore,
    \begin{align*}
      \lVert \nabla \cdot w \mu \nabla g
      \rVert_{L^s(I,X_{\vartheta+\frac{\epsilon}{2}})} &\leq \lVert
      w \mu \nabla g
      \rVert_{L^s(I,H^{\vartheta+\frac{\epsilon}{2},p})} \\ &\leq
      \sup_{t\in\overline I}\,\lVert w(t)\mu \rVert_{\Lkal(H^{\vartheta+\frac{\epsilon}{2},p})} 
      \lVert g
      \rVert_{L^s(I,H^{1+\vartheta+\frac{\epsilon}{2},p})} < \infty.
    \end{align*}
    The analogous argument applies when
    $\theta \geq \vartheta+\frac\eps2$. (The case $\theta = 1$ is
    special but straightforward since Lipschitz functions are
    multipliers on $H^{1,p}$.)
  \end{enumerate}
\end{remark}

\begin{remark}
  \label{rem::optimal-case} We comment on the optimal case $\Dom_{L^p}(-\nabla \cdot \mu \nabla + 1) = W^{2,p} \cap
    W^{1,p}_D$.
  \begin{enumerate}[(1)]
  \item It was already mentioned in Remark~\ref{rem:kappa-assumption}
    that in the optimal case, $\kappa = 2 - d/p$. Also, we can identify
    $Y_\theta$ in a more explicit manner. Indeed, as in
    Remark~\ref{rem::fracpowerass} we have
    \begin{equation*}
      Y_\theta = \bigl[W^{1,p}_D,W^{2,p} \cap W^{1,p}_D\bigr]_\theta =
      H^{1+\theta,p} \cap W^{1,p}_D.
    \end{equation*}
    Thus, $Y_\theta$ coincides with $H^{1+\theta,p}_D$ if $\theta < 1/p$. Moreover, recall that
    for $\theta \neq 1/p$, Lemma~\ref{lem::interpol-Bessel} tells us that $X_\theta$ is given
    precisely by $H^{\theta-1,p}_D$ and $H^{\theta-1,p}$ when $\theta < 1/p$ and $\theta> 1/p$,
    respectively. In the case $\theta = 1$, the condition on $s$ in
    Theorems~\ref{thm::eqinXtheta} and~\ref{thm::regforeq} becomes
    $\frac1s < \frac12 - \frac{d}{2p}$ which was the one posed in
    Assumption~\ref{ass::domain-geometry}, and which is also required for the starting point
    giving a global solution in the first place, Theorem~\ref{thm::exonW1p}. In this sense, it
    is optimal.
  \item In the optimal case, Theorem~\ref{thm::regforeq} with $\theta = 1$ essentially
    reproduces~\cite[Theorem~2.3]{Casas2018}.  In fact, we can even avoid the monotonicity
    assumption there and we do not have to require that the input data is $2s$-integrable in
    time to obtain an $s$-integrable solution. We concede that the assumption on $s$ is weaker
    in~\cite{Casas2018}, though; this is because the authors there do not pass through
    $\theta$-H\"older coefficients, and, thus, solutions, but only require \emph{continuous}
    ones. Note that there seems to be a mistake in the assumptions on time integrability
    in~\cite[Theorem~2.3]{Casas2018} as we have validated with the authors. In our present
    notation, the condition there should be $\frac1s < 1 - \frac{d}{2p}$ which is still less
    strict than ours. However, as mentioned above, for the optimal $\kappa = 2-d/p$ and
    $\theta = 1$, our requirement on $s$ falls back to the one required for
    Theorem~\ref{thm::exonW1p}, the starting point for the whole procedure. In this sense, it is
    optimal for our technique.
  \end{enumerate}
\end{remark}

\appendix

\section{Results related to interpolation theory}
\label{app::interpolation}

In the following we collect some auxiliary results related to
interpolation theory that are required in
Sections~\ref{sec::quasilinonXtheta} and~\ref{subsec::eqonXtheta}.

\subsection{Bilinear interpolation}
\label{sec::RehbergMeinlschmidt}

Let $A, B, X, Y$ be Banach spaces with continuous and dense inclusions
$A \embeds B$ and $X \embeds Y$. Further, let bounded
linear maps $\Phi\colon A \to \Lkal(X)$,
$\Psi\colon B \to \Lkal(Y)$ be given such that the following
compatibility property holds:
\begin{align}\label{compatibility}
  \text{for all}~a \in A\colon\quad \Psi(a)|_X = \Phi(a). 
\end{align}

\begin{theorem}\label{thm::interpolationthm}
  Adopt the above setting and~\eqref{compatibility}. Then, for every
  $\theta \in [0,1]$, the restriction of $\Psi$ to $[A,B]_\theta$
  gives rise to a bounded linear operator
  \[
    \Psi\colon [A,B]_\theta \to \Lkal\bigl([X,Y]_\theta\bigr).
  \]
\end{theorem}
\begin{proof}
  Consider the continuous bilinear map
  \[
    \ell\colon B \times Y \to Y, \qquad (b,y) \mapsto
    \ell(b,y) = 
    \Psi(b)y,
  \]
  and its restriction to $A \times X$ which coincides with
  $(a,x) \mapsto \Phi(a)x \in X$ on $A \times X$ due
  to~\eqref{compatibility}. We thus have the estimates
  \begin{align*}
    \bigl\lVert \ell(b,y) \bigr\rVert_Y &\leq \lVert \Psi(b)
    \rVert_{\Lkal(Y)} \lVert y \rVert_Y \leq \lVert \Psi
    \rVert_{\Lkal(B,\Lkal(Y))} \lVert b \rVert_B \lVert y \rVert_Y \intertext{and}
    \bigl\lVert \ell(a,x) \bigr\rVert_Y &\leq \lVert \Phi(a)
    \rVert_{\Lkal(X)} \lVert x \rVert_X \leq \lVert \Phi
    \rVert_{\Lkal(A,\Lkal(X))} \lVert a \rVert_A \lVert x \rVert_X. 
  \end{align*}
  Then, bilinear interpolation as in~\cite[Chapter~1.19.5]{Triebel1995}
  shows that $\ell$ gives rise to a continuous bilinear mapping
  $[A, B]_\theta \times [X,Y]_\theta \to [X,Y]_\theta$ whose norm is
  bounded by
  $\lVert \Phi \rVert_{\Lkal(A,\Lkal(X))}^{1-\theta} \lVert \Psi
  \rVert_{\Lkal(B,\Lkal(Y))}^\theta$. Hence, for $c \in
    [A,B]_\theta$ and $z \in [X,Y]_\theta$ we find
  \[
    \bigl\lVert \Psi(c) z \bigr\rVert_{[X,Y]_\theta} = \lVert \ell(c,z)
    \rVert_{[X,Y]_\theta} \leq \lVert \Phi
    \rVert_{\Lkal(A,\Lkal(X))}^{1-\theta} \lVert \Psi
    \rVert_{\Lkal(B,\Lkal(Y))}^\theta \cdot \lVert c
    \rVert_{[A,B]_\theta} \lVert z \rVert_{[X,Y]_\theta}.
  \]
  This shows that
  $\Psi \in \Lkal\bigl([A,B]_\theta, \Lkal([X,Y]_\theta)\bigr)$, as
  desired.
\end{proof}

\subsection{Interpolation of H\"older spaces}

The following result is well known for $E=\R^d$ or $E$ being a domain
with sufficiently smooth boundary, cf.\
e.g.~\cite{Lunardi1995,Triebel1995}. Since we did not find an explicit
reference for a less regular setting in the literature we decided to
sketch the short proof, although the appearing techniques are
standard.

\begin{proposition}\label{prop::interpolationhoelder}
  For any nonempty compact set $E \subseteq \R^d$ there holds
  \[
    \bigl(C(E),C^{0,1}(E)\bigr)_{\theta,\infty} = C^{0,\theta}(E), \qquad \theta
    \in (0,1),
  \]
  and
  \[
    C^{0,\vartheta}(E) \embeds \bigl[C(E),C^{0,1}(E)\bigr]_\theta
    \embeds C^{0,\theta}(E), \qquad \theta \in (0,1),~\vartheta \in (\theta,1).
  \]
\end{proposition}
 
\begin{proof}
  It suffices to prove the first statement. The second one follows
  from standard embeddings between real and complex interpolation
  scales. First, note that the claim holds for $E = \R^d$,
  see~\cite{Lunardi1995}, below the proof of Example~1.1.8, for
  instance. To extend this to general $E$, we use  the Whitney
  extension operator which provides a simultaneous
  extension operator $C(E) \to C(\R^d)$ and
  $C^{0,1}(E) \to C^{0,1}(\R^d)$,
  see~\cite[Chapter~VI.2]{Stein1970}. Then an application of the
  retraction-coretraction theorem~\cite[Theorem~1.2.4]{Triebel1995}
  finishes the proof.
\end{proof}

\subsection{Domains on interpolation spaces }

The following two results are used in the bootstrapping argument in
Section~\ref{subsec::eqonXtheta}.

\begin{lemma}\label{lem::domain1}
  Let $A$ be a closed operator on a Banach space $X$ with domain
  $\Dom_X(A)$ such that $A\colon \Dom_X(A) \to X$ is an
  isomorphism. Further, let $Y \embeds X$ be a subspace of $X$ such
  that $A$ is also an isomorphism $\Dom_Y(A) \to Y$. Then for
  $\theta \in (0,1)$ we have
  \[ \Dom_{[X,Y]_\theta}(A) = \bigl[\Dom_X(A),\Dom_Y(A)\bigr]_\theta.
  \]
\end{lemma}
\begin{proof}
  We equip $\Dom_Z(A)$ with the norm $x \mapsto \|Ax\|_Z$ for any
  occurring space $Z$. Then $A$ is an isometry between $\Dom_X(A)$ and
  $X$ and between $\Dom_Y(A)$ and $Y$. Thus, by the functorial
  property of interpolation, $A$ is nonexpansive, mapping
  $[\Dom_X(A),\Dom_Y(A)]_\theta \to [X,Y]_\theta$, i.e.,
  \begin{equation*}
    \|x\|_{\Dom_{[X,Y]_\theta}(A)} = \|Ax\|_{[X,Y]_\theta} \leq \|x\|_{[\Dom_X(A),\Dom_Y(A)]_\theta},
  \end{equation*}
  so
  \begin{equation*}
    [\Dom_X(A),\Dom_Y(A)]_\theta \embeds
    \Dom_{[X,Y]_\theta}(A).
  \end{equation*}
  Quite analogously, $A^{-1}$ is a contraction
  $[X,Y]_\theta \to [\Dom_X(A),\Dom_Y(A)]_\theta$, which shows
  that
  \begin{equation*}
    \Dom_{[X,Y]_\theta}(A) \embeds
    [\Dom_X(A),\Dom_Y(A)]_\theta.
  \end{equation*}
  The claim follows; in fact, the stated equality is even an isometry
  with the chosen norms.
\end{proof}
 

\begin{lemma}\label{lem::domain2}
  In addition to the assumptions of Lemma~\ref{lem::domain1}, assume
  that $A$ is densely defined, positive, and admits bounded imaginary
  powers, each both on $X$ and $Y$, respectively. Then for $\theta,\rho \in (0,1)$ there holds
  \[
    \left[[X,Y]_\theta, \Dom_{[X,Y]_\theta}(A)\right]_\rho =
    \Bigl[[X,\Dom_X(A)]_\rho, [Y,\Dom_Y(A)]_\rho\Bigr]_\theta.
  \]
\end{lemma}

\begin{proof} 
  By interpolation, $A$ has bounded imaginary powers on $[X,Y]_\theta$
  as well. With reiteration for fractional power domains as
  in~\cite[Theorem~1.15.3]{Triebel1995}, we thus conclude
  \[
    \bigl[[X,Y]_\theta, \Dom_{[X,Y]_\theta}(A)\bigr]_\rho =
    \Dom_{[X,Y]_\theta}(A^\rho).
  \]
  Now, according to~\cite[Theorem~1.15.2~(e)]{Triebel1995}, $A^\rho$
  is still an isomorphism $\Dom_{X}(A^\rho) \to X$ and
  $\Dom_{Y}(A^\rho) \to Y$, and hence fulfills the assumptions
  of the previous lemma. Therefore, we obtain
  \[
    \Dom_{[X,Y]_\theta}(A^\rho) = \bigl[\Dom_X(A^\rho),
    \Dom_Y(A^\rho)\bigr]_\theta.
  \]
  The claim follows from re-expanding the right-hand side
  via~\cite[Theorem~1.15.3]{Triebel1995}.
\end{proof}

\section{H\"older-regularity for the elliptic operator in \texorpdfstring{$L^p$}{Lp}}\label{sec:hold-regul-ellipt}

In relation to Definition~\ref{def::regdom} for the degree of
H\"older-regularity for solutions to the elliptic problem associated
to $-\nabla\cdot\mu\nabla + 1$ in $L^p$, we provide some examples for
explicit $\kappa$ in some rather nonsmooth cases:
 
\begin{example}\label{ex::convex-domain-regularity}
  Let $\Omega$ be a convex domain, let $\mu$ be Lipschitz-continuous,
  and let $\Gamma_D = \partial \Omega$ or $\Gamma_D =
  \emptyset$. From~\cite[Theorems 3.2.1.2\&3.2.1.3]{Grisvard1985} it
  is well known that
  \begin{align*}
    -\nabla \cdot \mu \nabla + 1\colon H^2 \cap \ker(B) \to L^2,
  \end{align*}  
  is a topological isomophism, with $B = \tr$ if
  $\Gamma_D = \partial \Omega$, or $B = \nu \cdot \mu \nabla$, the
  co-normal derivative on $\partial \Omega$ associated with $\mu$, if
  $\Gamma_D = \emptyset$. Hereby, we consider $\ker(B)$ as a subspaces
  of $W^{1,1}$ and $W^{2,1}$, respectively. Every convex domain is
  locally uniform~\cite[Proposition~3.8]{Triebel2008}. Consequently,
  there is a degree-independent Sobolev extension operator for
  $\Omega$~\cite{Rogers2006}, and the families
  $(L^q)_{q \in (1,\infty)}$ and
  $(W^{2,q} \mathrel\cap \ker(B) )_{q \in (1,\infty)}$ form
  interpolation scales such that
  \begin{align}\label{eq::W2qiso}
    -\nabla \cdot \mu \nabla + 1\colon W^{2,q} \cap \ker(B) \to L^q 
  \end{align}    
  is a bounded linear operator for every $q \in (1,\infty)$. It
  follows from Sneiberg's extrapolation theorem,
  see~\cite[Theorem~3.1]{Meinlschmidt2021}, that there is some
  $\bar q > 2$ such that~\eqref{eq::W2qiso} is still an isomorphism
  for $q \in [2,\bar q]$. Sobolev embeddings therefore imply that we
  can choose $\kappa > 1$ in dimension $d=2$ and
  $\kappa > \frac{1}{2}$ in dimension $d=3$.
\end{example}

The following example is concerned with polygonal domains in $\R^2$,
constant coefficients, and possibly mixed boundary conditions:

\begin{example}\label{ex::2d1}
  Let $\Omega \subset \R^2$ be a polygon with vertices $S_j$,
  $j=1,...,N$. By $\omega_j$ we denote the interior angle of $\Omega$
  at $S_j$. The type of boundary condition on each edge has to be
  fixed (homogeneous Dirichlet or Neumann), which is not a
  restriction, because ``artificial'' vertices with $\omega_j = \pi$
  are allowed. In order to keep the notation as simple as possible, we
  restrict ourselves to the Laplacian, i.e., $\mu = \id$, for this
  example.

  Given $g \in L^p$, let $f \in H^1_D$ denote the solution of
  $-\Delta f + f = g$. From~\cite[Theorem~4.4.3.7]{Grisvard1985} we infer that
  there is a regular part of the solution
  $f_0 \in W^{2,p}\cap W^{1,p}_D$ such that the remainder $f-f_0$ can
  be expressed in polar coordinates $(r_j,\theta_j)$ centered at $S_j$
  as follows:
  \begin{align}\label{eq::singular}
    (f-f_0)(r_j,\theta_j) = \sum_{1\leq j\leq N} \eta_j(r_j,\theta_j) \sum_{0 < \lambda_{j,m} < 2/p'} c_{j,m} r_j^{\lambda_{j,m}} s_{j,m}(\theta_j) 
  \end{align}
  Herein, $\eta_j$ denote smooth cut-off functions with their support
  near $S_j$, the $c_{j,m}$ are scalar coefficients, and the $s_{j,m}$ are smooth. The
  exponents $\lambda_{j,m}$ correspond to the eigenvalues
  of certain Sturm-Liouville eigenvalue problems, and can be
  determined analytically, cf.~\cite[p.220]{Grisvard1985}:
  \begin{align*}
    \lambda_{j,m} &= \frac{\pi m}{\omega_j} \qquad \text{for pure Dirichlet or Neumann boundary conditions around $S_j$}, \\
    \lambda_{j,m} &= \frac{\pi m}{2\omega_j} \qquad \text{if boundary conditions change at $S_j$.}
  \end{align*}
  Therefore, any $f \in \Dom_{L^p}(-\Delta + 1)$ consists of a
  $W^{2,p}$-part $f_0$, which is then in
  $C^{1,1-d/p}(\overline\Omega)$, and a singular part given
  by~\eqref{eq::singular}. From~\cite[Theorem~6.2.10]{Grisvard1985} we
  infer that this singular part is Hölder-continuous with its degree
  given by the minimal $\lambda_{j,m}$ appearing in the sum. Hence, we
  conclude
  \[
    \kappa = \min\left\{ 2-\frac{d}{p}, \min_{\overset{\omega_j >
          \pi}{ \textnormal{unif. BC's at
            $j$}}} \frac{\pi}{\omega_j}, \min_{\overset{\omega_j >
          \frac{\pi}{2}}{ \textnormal{mixed BC's at
            $j$}}} \frac{\pi}{2\omega_j} \right\}
  \]
  in this example.  Note that using~\cite[Theorem 5.2.7]{Grisvard1985}
  also non-constant, but Lipschitz-continuous $\mu$ can be treated in
  case of pure homogeneous Dirichlet boundary conditions.
\end{example}

Following~\cite{Petzoldt2001}, Example~\ref{ex::2d1} can be extended
to a \emph{scalar} coefficient function $\mu$ that is constant with
respect to a poylgonal partition of $\Omega$:

\begin{example}\label{ex::2d2}
  Let $\Omega \subset \R^2$ be a polygon and $\mu$ be a scalar
  function that is piecewise constant with respect to a partition of
  $\Omega$ into subdomains given by
  polygons. In~\cite[Theorem~2.27]{Nicaise1993} a decomposition of
  $f \in \Dom_{L^p}(-\nabla \cdot \mu \nabla + 1)$ into a
  $W^{2,p}$-part and a singular part, similar to the previous example,
  is obtained. If $\mu$ is non-constant, then singularities can occur
  at any vertex of the boundary and also along the discontinuities of $\mu$
  inside the domain. Again, the ``bad'' contributions to $f$
  arise from terms of type $r^\lambda s(\theta)$ (in polar
  coordinates) with $\lambda$ to be determined from certain
  Sturm-Liouville eigenvalue problems, cf.~\cite{Petzoldt2001}.  In
  order to have lower bounds on these eigenvalues depending only on
  the underlying geometry, but not on the actual values of $\mu$, we
  require the following restriction: For any $x \in \bar \Omega$ the
  number of different values of $\mu$ (``materials'') adjacent to $x$
  plus the number of boundary conditions at $x$ does not exceed
  3. Under this assumption the following eigenvalue estimates are
  provided in~\cite{Petzoldt2001}, see also~\cite[Section~8]{Costabel1999}:
  \begin{align*}
      \lambda_x &> \frac{\pi}{2\omega_x} & &x \in \partial \Omega, \text{ two materials, uniform boundary condition}, \\
      \lambda_x &> \frac{\pi}{2\omega_x} & &x \in \partial \Omega, \text{ one material, boundary condition changing}, \\
      \lambda_x &> \frac{1}{2} & &x \in \Omega, \text{ two adjacent materials}, \\
      \lambda_x &> \frac{1}{4} & &x \in \Omega, \text{ three
        adjacent materials}.
  \end{align*}
  Hereby, $\omega_x$ denotes the interior angle of $\Omega$ at the
  (possibly artificial) vertex $x \in \partial \Omega$. The remaining
  case, i.e., a point $x \in \partial \Omega$ with only one adjacent
  material and uniform boundary condition, has already been dealt with
  in Example~\ref{ex::2d1}. For the details and similar results under
  less strict assumptions we refer to~\cite{Petzoldt2001}. As in the
  previous example, the maximal value for $\kappa$ is determined by
  the minimal $\lambda_x$ for $x \in \overline \Omega$.
\end{example}

In three space dimensions the situation is more involved. For example,
in elliptic problems on polyhedral domains there are singularities
arising from both vertices and edges. We refer
to~\cite{Dauge1992,Dauge1999,Grisvard1995} for an overview. For a
constant coefficient function $\mu$ and pure homogeneous Dirichlet
boundary conditions, a decomposition into a regular and a singular
part similar to the one in Example~\ref{ex::2d1} is still possible,
see e.g. the exposition in the proof
of~\cite[Lemma~2.3]{Wollner2012}. The resulting degree of
Hölder-continuity, however, depends on the eigenvalues of the
Laplace-Beltrami operator on certain spherical polygons and seems to
be difficult to determine. We thus do not go into further details.





\section*{Acknowledgements}

The authors thank J.~Rehberg (WIAS, Berlin) for initiating early work on
this project, encouragement, and helpful discussions.

\printbibliography

\end{document}